\documentclass[reqno,a4paper,12pt]{amsart}

\usepackage{mystyle}
\usepackage{etoolbox}

\usepackage[utf8]{inputenc}
\usepackage[T1]{fontenc}
\usepackage{lmodern}
\usepackage[british]{babel}
\usepackage[nopatch=eqnum]{microtype}
\usepackage[headings]{fullpage}
\usepackage{parskip}

\usepackage[pdfencoding=auto,pdfusetitle]{hyperref}
\hypersetup{colorlinks=true, linkcolor=blue!30!black, citecolor=green!30!black, urlcolor=red!45!black}

\usepackage{comment}

\usepackage{amsfonts,amssymb,bbm,stmaryrd}
\usepackage[scr=boondoxo,scrscaled=1.05]{mathalfa}
\usepackage{amsmath}

\usepackage{mathtools,thmtools}
\hypersetup{hypertexnames=false}
\usepackage[noabbrev,capitalise]{cleveref}
\crefformat{equation}{(#2#1#3)}
\crefformat{enumi}{(#2#1#3)}
\crefformat{enumii}{(#2#1#3)}

\usepackage{centernot,float,subcaption,url,array,booktabs,colortbl}

\usepackage[shortlabels]{enumitem}
\setlist[enumerate,1]{label={\roman*)}}

\usepackage{tikz}
\usepackage{tikz-cd}
\usetikzlibrary{spath3,intersections,calc}
\usetikzlibrary{decorations.pathmorphing}
\usetikzlibrary{matrix,positioning}
\usetikzlibrary{shapes.geometric}
\usetikzlibrary{arrows}
\tikzset{> =stealth}

\usepackage{xparse}
\usepackage{adjustbox}
\usepackage{multirow,array,longtable}
\usepackage{placeins,xspace,fix-cm}
\usepackage{accents}
\usepackage[nomessages]{fp}

\newcolumntype{F}{>{$}c<{\hspace{-0.9ex}$}}
\newcolumntype{:}{>{$}m{0.8ex}<{$}}
\newcolumntype{R}{>{$}r<{$}}
\newcolumntype{C}{>{$}c<{$}}
\newcolumntype{L}{>{$}l<{$}}
\newcolumntype{N}{@{}>{$}l<{$}}
\newlength\horspace
\newcommand{\h}[1][1.0]{\hspace*{#1\horspace}}
\newlength\verspace
\renewcommand{\v}[1][1.0]{\vspace*{#1\verspace}\xspace}
\tikzset{iso/.style={draw=none,every to/.append style={edge node={node [sloped, allow upside down, auto=false]{$\cong$}}}}}
\tikzset{adj/.style={draw=none,every to/.append style={edge node={node [sloped, allow upside down, auto=false]{$\dashv$}}}}}
\tikzset{simeq/.style={draw=none,every to/.append style={edge node={node [sloped, allow upside down, auto=false]{$\simeq$}}}}}
\tikzset{simeqS/.style={draw=none,every to/.append style={edge node={node [sloped, allow upside down, auto=false]{$\raisebox{0.8em}{$\simeq$}$}}}}}
\tikzset{aiso/.style={simeqS,preaction={draw,->}}}
\tikzset{RightA/.style={double distance=3.5pt,>={Implies},->},%
	triple/.style={-,preaction={draw,RightA}},%
	quadruple/.style={preaction={draw,RightA,shorten >=0pt},shorten >=1pt,-,double,double distance=0.2pt}}
\tikzset{curve/.style={settings={#1},to path={(\tikztostart)
    .. controls ($(\tikztostart)!\pv{pos}!(\tikztotarget)!\pv{height}!270:(\tikztotarget)$)
    and ($(\tikztostart)!1-\pv{pos}!(\tikztotarget)!\pv{height}!270:(\tikztotarget)$)
    .. (\tikztotarget)\tikztonodes}},
    settings/.code={\tikzset{quiver/.cd,#1}
        \def\pv##1{\pgfkeysvalueof{/tikz/quiver/##1}}},
    quiver/.cd,pos/.initial=0.35,height/.initial=0}

\newcommand{\addQEDstyle}[2]{\AtBeginEnvironment{#1}{\pushQED{\qed}\renewcommand{\qedsymbol}{#2}}\AtEndEnvironment{#1}{\popQED}}

\theoremstyle{plain}
\newtheorem{theorem}{Theorem}[section]
\newtheorem{lemma}[theorem]{Lemma}
\newtheorem{proposition}[theorem]{Proposition}
\newtheorem{corollary}[theorem]{Corollary}

\theoremstyle{definition}
\newtheorem{definition}[theorem]{Definition}

\theoremstyle{remark}
\newtheorem{remark}[theorem]{Remark}
\newtheorem{cons}[theorem]{Construction}

\addQEDstyle{example}{$\triangle$}

\def\nameit#1{\textrm{#1}~}
\def\thex{\nameit{Theorem}}
\def\prox{\nameit{Proposition}}

\def\defx{\nameit{Definition}}
\def\remx{\nameit{Remark}}
\def\consx{\nameit{Construction}}

\def\dfn#1{{\itshape #1}}

\NewDocumentEnvironment{cd}{s O{5.5} O{5.5} b}{%
	\IfBooleanF{#1}{\begin{equation*}}\begin{tikzcd}[row sep=#2ex,column sep=#3ex,ampersand replacement=\&]
			#4
		\end{tikzcd}\IfBooleanF{#1}{\end{equation*}}\ignorespacesafterend}{}

\newenvironment{fun}{\[\begin{tabular}{F:RCL}}{\end{tabular}\]\ignorespacesafterend}

\newenvironment{eqD*}{\begin{equation*}}{\end{equation*}\ignorespacesafterend}

\renewcommand{\epsilon}{\varepsilon}
\renewcommand{\phi}{\varphi}

\mathchardef\mhyphen="2D

\newcommand{\Z}{\mathbb{Z}}

\newcommand{\id}{\mathrm{id}}
\newcommand{\invv}{i}

\newcommand{\op}{{^\mathrm{\hspace{0.5pt}op}}}

\def\:{\colon}

\newcommand{\p}[1]{\textup{(}{#1}\kern2pt\textup{)}}

\def\c{\circ}
\newcommand{\iso}{\cong}

\newcommand{\scaleu}[2][1.2]{{\scalebox{#1}{$#2$}}}

\newcommand{\ov}[1]{\overline{#1}}

\newcommand{\stb}{_{\ast}}

\DeclareFontFamily{OT1}{pzc}{}
\DeclareFontShape{OT1}{pzc}{m}{it}{<->s*[1.19]pzcmi7t}{}
\DeclareMathAlphabet{\mathpzc}{OT1}{pzc}{m}{it}

\DeclareFontFamily{U}{dutchcal}{\skewchar\font=45}
\DeclareFontShape{U}{dutchcal}{m}{n}{<->s*[1.05] dutchcal-r}{}
\DeclareMathAlphabet{\mathlcal}{U}{dutchcal}{m}{n}
\newcommand{\catfont}[1]{\ensuremath{\mathpzc{#1}}\xspace}

\newcommand{\A}{\catfont{A}}
\newcommand{\B}{\catfont{B}}
\newcommand{\C}{\catfont{C}}
\newcommand{\D}{\catfont{D}}
\newcommand{\E}{\catfont{E}}
\newcommand{\M}{\catfont{M}}

\newcommand{\K}{\catfont{K}}
\newcommand{\X}{\catfont{X}}
\newcommand{\Y}{\ensuremath{\mathcal{Y}}\xspace}
\newcommand{\1}{\catfont{1}}

\newcommand{\Set}{\catfont{Set}}
\newcommand{\Grp}{\catfont{Grp}}
\newcommand{\Ab}{\catfont{Ab}}
\newcommand{\Mon}{\catfont{Mon}}
\newcommand{\Cliff}{\catfont{Cliff}\!}

\newcommand{\CInvMon}{\catfont{CInvMon}}
\newcommand{\SLat}{\catfont{SLat}}

\newcommand{\Cat}{\catfont{Cat}}
\newcommand{\ICat}{\catfont{ICat}}
\newcommand{\OpICat}{\catfont{OpICat}}
\newcommand{\ISet}{\catfont{ISet}}
\newcommand{\SymMonCat}
{\catfont{SymMonCat}}
\newcommand{\MonCat}
{\catfont{MonCat}}
\newcommand{\CartMonCat}
{\catfont{CartMonCat}}
\newcommand{\MonCatlax}
{\catfont{MonCat}_{\lax}}
\newcommand{\TwoCatps}{\catfont{2Cat}_{\ps}}
\newcommand{\MonTwoCat}{\catfont{Mon2Cat}}
\newcommand{\MonTwoCatps}{\catfont{Mon2Cat}_{\ps}}
\newcommand{\Bicat}{\catfont{Bicat}}
\DeclareMathOperator{\Monf}{Mon}
\DeclareMathOperator{\Grpf}{Grp}
\DeclareMathOperator{\Abf}{Ab}
\DeclareMathOperator{\PsMon}{PsMon}
\newcommand{\twoCat}{2\h[2]\mbox{-}\Cat}
\NewDocumentCommand{\Fib}{t' t" t+ t? O{n} O{n} o}{
	\ensuremath{\ifx#5t{2\mbox{-}\Set\mbox{-}}\fi\catfont{\ifx#5d{D}\else{\ifx#5o{Op}\else{\ifx#5b{DOp}\else{\ifx#5t{Op}\else{\ifx#5c{Cl\h[-3]}\else{\ifx#5s{Sp}\fi}\fi}\fi}\fi}\fi}\fi{Fib}}\IfBooleanT{#3}{^{\h[3.7]\opn{s}\h[-3.2]}}\IfBooleanT{#4}{^{\h[3.7]\opn{P}\h[-3,7]}}{\IfBooleanTF{#1}{_{\h[0.4]\opn{cart}\ifx#6n{}\else{\h[-1.4],\h[0.4]{#6}}\fi}}{\IfBooleanTF{#2}{_{\h[0.4]\opn{clov}\ifx#6n{}\else{\h[-1.4],\h[0.4]{#6}}\fi}}{\ifx#6n{}\else{_{\h[0.4]{#6}}}\fi}}}\IfNoValueF{#7}{\h[-1]\left({#7}\right)}}
}

\NewDocumentCommand{\Sh}{o m}{
	\ensuremath{\catfont{Sh}\hspace{-0.15ex}\left({#2}\IfNoValueF{#1}{,{#1}}\right)}
}

\NewDocumentCommand{\Alg}{t+ O{n} m}{
	#3\mbox{-}\catfont{\IfBooleanT{#1}{Ps}Alg}\ifx#2l{_{\lax}}\else{\ifx#2p{_{\ps}}\else{\ifx#2o{_{\oplax}}\fi}\fi}\fi}

\newcommand{\Lan}[1]{\operatorname{Lan}_{\hspace{0.2ex}#1}}

\newcommand{\opn}[1]{\operatorname{#1}}

\newcommand{\HomC}[3]{{#1}\!\left({#2},\h[1]{#3}\right)}

\newcommand{\x}[1][]{\h[-1]\times_{#1}\h[-1]}
\newcommand{\xf}{\otimes_F}
\newcommand{\xgf}{\otimes_{\Groth{F}}}
\newcommand{\xfm}{\otimes_{F(M)}}
\newcommand{\restr}[2]{{\left.\kern-\nulldelimiterspace {#1}\vphantom{\big|} \right|_{#2}}}

\newcommand{\cod}{\operatorname{cod}}

\DeclareMathOperator{\lax}{lax}
\DeclareMathOperator{\oplax}{oplax}

\DeclareMathOperator{\ps}{ps}

\newcommand{\Int}[1]{\ensuremath{\int \hspace{-0.35ex} #1}}
\newcommand{\Intdiag}[1]{\ensuremath{\scaleu{\int} \hspace{-0.15ex} #1}}

\newcommand{\Groth}[1]{\Int{{}#1}}
\newcommand{\Grothdiag}[1]{\Intdiag{#1}}

\newcommand{\groth}[1]{\mathcal{G}\mkern-1.4mu\left(#1\right)}

\newcommand{\too}{\longrightarrow}
\newcommand{\mto}{\mapsto}

\makeatletter
\newcommand{\ar}[2][]{\xrightarrow[#1]{#2}}
\newcommand{\aar}[2][]{\xrightarrow[#1]{#2}}
\def\xlongrightarrowfill@{\arrowfill@\relbar\relbar\longrightarrow}
\newcommand{\arr}[2][]{%
	\ext@arrow 0099\xlongrightarrowfill@{#1}{#2}}
\newcommand{\aarr}[2][]{%
	\ext@arrow 0099\xlongrightarrowfill@{#1}{#2}}
\def\xprorightarrowfill@{\arrowfill@{\relbar\joinrel\raisebox{0.6pt}{\small$\shortmid$}\joinrel{\relbar}}\relbar\rightarrow}
\newcommand{\aproarrow}[2][]{%
	\ext@arrow 0099\xprorightarrowfill@{#1}{#2}}

\newcommand{\aR}[2][]{%
	\ext@arrow 0055{\Rightarrowfill@}{#1}{#2}}
\def\xLongrightarrowfill@{\arrowfill@\Relbar\Relbar\Longrightarrow}
\newcommand{\aRR}[2][]{%
	\ext@arrow 0099\xLongrightarrowfill@{#1}{#2}}
\def\aitofill@{\arrowfill@{\lhook\joinrel\relbar}\relbar\rightarrow}
\newcommand{\aito}[2][]{%
	\ext@arrow 3095\aitofill@{#1}{#2}}
\def\Longaitofill@{\arrowfill@{\lhook\joinrel\relbar\joinrel\relbar}\relbar\rightarrow}
\newcommand{\aitoo}[2][]{%
	\ext@arrow 0099\Longaitofill@{#1}{#2}}

\def\xlongleftarrowfill@{\arrowfill@\longleftarrow\relbar\relbar}
\newcommand{\all}[2][]{%
	\ext@arrow 0099\xlongleftarrowfill@{#1}{#2}}
\newcommand{\aL}[2][]{%
	\ext@arrow 0055{\Leftarrowfill@}{#1}{#2}}
\def\xLongleftarrowfill@{\arrowfill@\Longleftarrow\Relbar\Relbar}
\newcommand{\aLL}[2][]{%
	\ext@arrow 0099\xLongleftarrowfill@{#1}{#2}}
\def\xmapstofill@{\arrowfill@{\mapstochar\relbar}\relbar\rightarrow}
\newcommand{\am}[2][]{%
	\ext@arrow 0395\xmapstofill@{#1}{#2}}
\def\xlongmapstofill@{\arrowfill@\relbar\relbar\longmapsto}
\newcommand{\amm}[2][]{%
	\ext@arrow 0399\xlongmapstofill@{#1}{#2}}

\newcommand{\eqq}{\DOTSB\protect\Relbar\protect\joinrel\Relbar}
\def\xeqqfill@{\arrowfill@\Relbar\Relbar\eqq}
\newcommand{\aeqq}[2][]{%
	\ext@arrow 0099\xeqqfill@{#1}{#2}}
\def\xRrightarrowfill@{\arrowfill@\equiv\equiv\Rrightarrow}
\newcommand{\aM}[2][]{\ext@arrow 0359\xRrightarrowfill@{#1}{#2}}
\newcommand{\Llongrightarrow}{%
	\DOTSB\protect\equiv\protect\joinrel\Rrightarrow}
\def\xLlongrightarrowfill@{\arrowfill@\equiv\equiv\Llongrightarrow}
\newcommand{\aMM}[2][]{%
	\ext@arrow 0099\xLlongrightarrowfill@{#1}{#2}}
\makeatother

\newcommand{\aequi}{\ensuremath{\stackrel{\raisebox{-1ex}{\kern-.3ex$\scriptstyle\sim$}}{\rightarrow}}}
\newcommand{\aequii}{\ensuremath{\stackrel{\raisebox{-1ex}{\kern-.3ex$\scriptstyle\sim$}}{\longrightarrow}}}

\newcommand{\PB}[1]{\arrow[#1,phantom,"\scalebox{1.6}{\color{black}$\lrcorner$}",very near start]}
\newcommand{\Ar}[4][]{\arrow[#2,"{#3}"{#1},""{name=#4, anchor=center}]}
\newcommand{\Ars}[4][]{\arrow[#2,"{#3}"'{#1},""{name=#4, anchor=center}]}
\newcommand{\Arb}[6][]{\arrow[#2,"{#3}"{#1},from=#4,to=#5,shorten <= #6 em, shorten >= #6 em]}
\newcommand{\Arbs}[6][]{\arrow[#2,"{#3}"'{#1},from=#4,to=#5,shorten <= #6 em, shorten >= #6 em]}

\NewDocumentCommand{\fib}{O{n} O{2.3} mmm}{%
	\begin{cd}*[#2][5]
		{#3}\ifx#1n{\arrow[d,"{\,\scaleu{#4}}"]}\else{\ifx#1i{\arrow[d,hookrightarrow,"{\,\scaleu{#4}}"]}\else{\ifx#1e{\arrow[d,equal,"{\,\scaleu{#4}}"]}\else{\ifx#1R{\arrow[d,Rightarrow,"{\,\scaleu{#4}}"]}\fi}\fi}\fi}\fi\\
		{#5}\ifx#1o{\arrow[u,"{\,\scaleu{#4}}"']}\fi
	\end{cd}\xspace
}
\NewDocumentCommand{\fibdiag}{O{n} O{2.3} mmm}{%
	\begin{cd}*[#2][5]
		{#3}\ifx#1n{\arrow[d,"{\,{#4}}"]}\else{\ifx#1i{\arrow[d,hookrightarrow,"{\,{#4}}"]}\else{\ifx#1e{\arrow[d,equal,"{\,{#4}}"]}\else{\ifx#1R{\arrow[d,Rightarrow,"{\,{#4}}"]}\fi}\fi}\fi}\fi\\
		{#5}\ifx#1o{\arrow[u,"{\,{#4}}"']}\fi
	\end{cd}\xspace
}
\NewDocumentCommand{\sq}{s O{n} O{7} O{7} O{} O{2.7} O{2.2} O{0.5} O{n}}{%
	\def\foosq##1##2##3##4##5##6##7##8{%
		\IfBooleanTF{#1}{\begin{cd}*}{\begin{cd}}[#3][#4]
				{##1}\ifx#2p{\PB{rd}}\fi\arrow[r,"{##5}"]\ifx#9l{\arrow[d,equal,"{##6}"']}\else{\arrow[d,"{##6}"']}\fi\&{##2}\ifx#9r{\arrow[d,equal,"{##7}"]}\else{\arrow[d,"{##7}"]}\fi\ifx#2l{\arrow[ld,Rightarrow,shorten <=#6ex,shorten >=#7ex,"{#5}"{pos=#8}]}\fi\\
				{##3}\ifx#9d{\arrow[r,equal,"{##8}"']}\else{\arrow[r,"{##8}"']}\fi\ifx#2o{\arrow[ur,Rightarrow,shorten <=#6ex,shorten >=#7ex,"{#5}"{pos=#8}]}\fi\&{##4}
		\end{cd}}%
		\foosq}
\newcommand{\fibsq}[8][6]{%
	\begin{cd}[5.85][#1]
		#2 \arrow[r,"{#6}"] \arrow[d,mapsto,"{#7}"'] \& #3 \arrow[d,mapsto,"{#7}"] \\
		#4 \arrow[r,"{#8}"'] \& #5
\end{cd}}

\newcommand{\fibsquniv}[8][6]{%
	\def\foofibsquniv##1##2##3##4##5{%
		\begin{cd}[3][#1]
			\& \&[-3ex] ##1 \arrow[dd,mapsto,"{#7}"]\\[-4.5ex]
			#2 \arrow[rru,bend left,"{##3}"]\arrow[r,"{#6}"] \arrow[dd,mapsto,"{#7}"'] \& #3\arrow[ru,dashed,"{##4}"] \arrow[dd,mapsto,"{#7}"] \\
			\& \& ##2 \\[-4.5ex]
			#4 \arrow[r,"{#8}"'] \& #5 \arrow[ru,"{##5}"']
			\end{cd}}%
	\foofibsquniv%
}
\NewDocumentCommand{\nat}{s O{n} O{7} O{7} O{2.7} O{2.2} O{0.5} O{n}}{%
	\def\foonat##1##2##3##4##5##6{%
		\IfBooleanTF{#1}{\sq*}{\sq}[#2][#3][#4][{##1}_{##4}][#5][#6][#7][#8]{{##2}\ifx#8l{}\else{({##5})}\fi}{{##3}\ifx#8r{}\else{({##5})}\fi}{{##2}\ifx#8l{}\else{({##6})}\fi}{{##3}\ifx#8r{}\else{({##6})}\fi}{{##1}_{##5}}{\ifx#8l{}\else{{##2}({##4})}\fi}{\ifx#8r{}\else{{##3}({##4})}\fi}{{##1}_{##6}}}%
	\foonat}
\NewDocumentCommand{\tr}{s O{4.5} O{6.5} O{0} O{0} O{n} O{0} O{} O{0}}{%
	\def\footr##1##2##3##4##5##6{%
		\IfBooleanTF{#1}{\begin{cd}*}{\begin{cd}}[#3][#2]
				{##1}\arrow[rr,"{##4}"]
				\Ars[inner sep =0.2ex]{dr}{##5}{A}\&[#4ex]\&[#5ex]{##2}\Ar[inner sep =0.2ex]{ld}{##6}{B}\\
				\&{##3}
				\ifx#6l{\Arb{Rightarrow,shift right=#7em}{#8}{A}{B}{#9}}\else{\ifx#6o{\Arbs{Rightarrow,shift right=#7em}{#8}{B}{A}{#9}}\else{\ifx#6i{\Arbs[inner sep=0.9ex]{iso,shift right=#7em}{#8}{A}{B}{#9}}\else{\ifx#6e{\Arb{equal,shift right=#7em}{#8}{A}{B}{#9}}\else{}\fi}\fi}\fi}\fi
		\end{cd}}%
		\footr}

\NewDocumentCommand{\tc}{s t+ O{7} O{30} O{} O{} O{} o}{
	\def\footc##1##2##3##4##5{%
		\FPmul\Mulresulttwo{#3}{#3}%
		\FPmul\Mulresult{0.0026}{\Mulresulttwo}%
		\IfBooleanTF{#1}{\begin{cd}*}{\begin{cd}}[#3][#3]
				{##1}\Ar[#5]{r,bend left=#4}{##3}{A}\Ars[#6]{r,bend right=#4}{##4}{B}\&{##2}
				\IfBooleanTF{#2}{\Arb[description,pos=0.49]}{\Arb}{Rightarrow #7}{\mkern1mu {##5}}{A}{B}{\IfNoValueTF{#8}{\Mulresult}{#8}}
		\end{cd}}%
		\footc}
\NewDocumentCommand{\tcwl}{s t+ O{7} O{30} O{} O{} O{} o O{-2}}{
	\def\footcwl##1##2##3##4##5##6##7{%
		\FPmul\Mulresulttwo{#3}{#3}%
		\FPmul\Mulresult{0.0026}{\Mulresulttwo}%
		\IfBooleanTF{#1}{\begin{cd}*}{\begin{cd}}[#3][#3]
				##6 \arrow[r,"{##7}"]\&[#9ex]{##1}\Ar[#5]{r,bend left=#4}{##3}{A}\Ars[#6]{r,bend right=#4}{##4}{B}\&{##2}\IfBooleanTF{#2}{\Arb[description,pos=0.49]}{\Arb}{Rightarrow #7}{\mkern1mu {##5}}{A}{B}{\IfNoValueTF{#8}{\Mulresult}{#8}}
		\end{cd}}%
		\footcwl}
\NewDocumentCommand{\tcwr}{s t+ O{7} O{30} O{} O{} O{} o O{-2}}{
	\def\footcwr##1##2##3##4##5##6##7{%
		\FPmul\Mulresulttwo{#3}{#3}%
		\FPmul\Mulresult{0.0026}{\Mulresulttwo}%
		\IfBooleanTF{#1}{\begin{cd}*}{\begin{cd}}[#3][#3]
				{##1}\Ar[#5]{r,bend left=#4}{##3}{A}\Ars[#6]{r,bend right=#4}{##4}{B}\&{##2}\arrow[r,"{##7}"]\&[#9ex]##6\IfBooleanTF{#2}{\Arb[description,pos=0.49]}{\Arb}{Rightarrow #7}{\mkern1mu {##5}}{A}{B}{\IfNoValueTF{#8}{\Mulresult}{#8}}
		\end{cd}}%
		\footcwr}
\NewDocumentCommand{\tcv}{s t' O{7} O{30} mmmmm}{
	\FPmul\Mulresulttwo{#3}{#3}%
	\FPmul\Mulresult{0.0026}{\Mulresulttwo}%
	\IfBooleanTF{#1}{\begin{cd}*}{\begin{cd}}[#3][#3]
			{#5}\IfBooleanTF{#2}{\Ars{d,leftarrow,bend right=#4}{#7}{A}\Ar{d,leftarrow,bend left=#4}{#8}{B}}{\Ars{d,bend right=#4}{#7}{A}\Ar{d,bend left=#4}{#8}{B}}\\{#6}
			\Arb{Rightarrow}{#9}{A}{B}{\Mulresult}
		\end{cd}}

\title{Clifford semigroups and the monoidal Grothendieck construction}

\author[E. Caviglia]{E.\ Caviglia}
\address{Department of Mathematical Sciences, Stellenbosch University, South Africa
\newline
National Institute for Theoretical and Computational Sciences, Stellenbosch, South Africa}
\email{elena.caviglia@outlook.com}

\author[P. F. Faul]{P.\ F.\ Faul}
\address{Department of Mathematical Sciences, Stellenbosch University, South Africa
\newline
National Institute for Theoretical and Computational Sciences, Stellenbosch, South Africa}
\email{peter@faul.io}

\author[G. Manuell]{G.\ Manuell}
\address{Department of Mathematical Sciences, Stellenbosch University, South Africa
\newline
National Institute for Theoretical and Computational Sciences, Stellenbosch, South Africa}
\email{graham@manuell.me}

\author[L. Mesiti]{L.\ Mesiti}
\address{Department of Mathematical Sciences, Stellenbosch University, South Africa}
\email{luca.mesiti@outlook.com}

\date{August 2026}


\begin{document}

\begin{abstract}
    We show that the monoidal Grothendieck construction can be applied to recover the known structure theorem for Clifford monoids, which states that they correspond to functors from a semilattice into the category of groups. Furthermore, we capture the category of Clifford monoids itself as a Grothendieck construction of the functor sending a semilattice $L$ to the functor category $\HomC{\Cat}{L\op}{\Grp}$ and use this to construct a number of factorisation systems on this category. Finally, we prove a general result on taking monoids in a monoidal fibration and apply it to establish a correspondence between inverse semirings and lax monoidal functors from an idempotent semiring into the category of abelian groups.
\end{abstract}

\maketitle
\thispagestyle{empty}

\section{Introduction}\label{sec:introduction}

The Grothendieck construction exhibits one of the most fundamental equivalences in category theory. In its base form it gives a correspondence between Grothendieck fibrations and indexed categories, but it has subsequently been expanded and generalised to broader settings. In this paper we will use one of these generalisations to clarify some aspects of the theory of Clifford monoids and prove a new structure theorem for inverse semirings.

A Clifford semigroup is a semigroup $M$ equipped with a unary involutive `inverse' operation $(-)^*$ such that $xx^*$ acts like an identity for $x$ in the sense that $xx^*x = x$ and commutes with every element of $M$. Equivalently, a Clifford semigroup is an inverse semigroup in which $xx^* = x^*x$ (see \cite[Chapter 5]{lawson1998inverse}) or an inverse semigroup in which idempotents are central. We will find it slightly more convenient to work with Clifford \emph{monoids}, though all of the results can be extended to Clifford semigroups without much effort.
We denote the category of Clifford monoids and monoid homomorphisms by $\Cliff$. (Note that any monoid homomorphism between Clifford monoids automatically preserves inverses.)

If $x$ is an element of a Clifford monoid $M$, then $xx^*$ is an idempotent. We write $e_x = xx^*$. The idempotents of $M$ form a meet-semilattice under multiplication called $E(M)$. The construction extends to a functor in the obvious way.
Clifford monoids have a natural order structure, distinct from the usual divisibility order on monoids. We say $x \le y$ if there exists an idempotent $e$ such that $x = ye$. In fact, $e$ may always be taken to be $e_x$. This order agrees with the natural order on the idempotents.
The elements of $M$ associated to a given idempotent are closed under multiplication and actually form a group with unit $e$. We call this group $G_e$.

The structure of Clifford monoids is well-understood \cite{clifford1941semigroups}. Every Clifford monoid $M$ gives rise to a functor $G_M\colon E(M)\op \to \Grp$, where $E(M)$ is viewed as a category under its natural order, $G_M(e) = G_e$, and $G_M(e \le e')\colon g \mapsto g e$. Then $M$ may be recovered as the set $\{ (e,g) \mid e \in E(M), g \in G_M(e) \}$ with $(e,g) \cdot (e',g') = (e e', G_M(ee' \le e)(g) \cdot G_M(ee' \le e')(g'))$. Perhaps less well-known is that this correspondence extends to an equivalence of categories (see \cite{pasku2011clifford}). 

This reconstruction of the Clifford monoids as sets of pairs $(e,x)$ where $x \in G_M(e)$ is reminiscent of the Grothendieck construction. Since the Grothendieck construction takes in a (pseudo)functor and returns a category, some work must be done to make this connection precise. If we treat the groups $G_M(e)$ as discrete monoidal categories then the \emph{monoidal} Grothendieck construction (see \cite{MoeVas20}) yields a category whose objects are elements of $M$, whose morphisms are given by the well-known order structure of inverse semigroups and which has a monoidal product corresponding to the multiplication on $M$.

But this is not where the story ends. We will show that the category of all Clifford monoids described in \cite{pasku2011clifford} itself arises from a Grothendieck construction applied to a pseudofunctor sending a semilattice $L$ to the functor category $\HomC{\Cat}{L\op}{\Grp}$.

Recall that the Grothendieck construction relates pseudofunctors into $\Cat$ (or \emph{indexed categories}) and fibrations.
In fact, the resulting functor $E\colon \Cliff \to \SLat$ is both a fibration and opfibration. This fact can be leveraged to lift factorisation systems from $\SLat$ to $\Cliff$. In this way we obtain four factorisation systems for $\Cliff$, three of which have very natural interpretations in terms of known concepts in the study of inverse semigroups.

Finally, we turn our attention to inverse semirings. These are semirings which form a commutative inverse monoid under addition. These structures arise naturally for computational and topological reasons and an overview of their theory can be found in \cite{faul2026}. Just as Clifford monoids can be understood in terms of functors from semilattices into $\Grp$, we might seek a similar characterisation of inverse semirings. In fact, we will show that they correspond to lax monoidal functors from idempotent semirings into $\Ab$.

We reach this result by proving a general theorem on taking monoids in the monoidal Grothendieck construction. More precisely, we show that given a monoidal opfibration, the induced functor between the categories of monoids can itself be obtained by a Grothendieck construction of a pseudofunctor that takes monoids in the fibres. We then apply this theorem to commutative Clifford monoids, which are the same thing as commutative inverse monoids. Indeed, there is a natural tensor product on this category and the restriction of $E\colon \Cliff \to \SLat$ becomes a monoidal opfibration. And its presheaf counterpart becomes a (weakly) lax monoidal pseudofunctor from $\SLat$ to $\Cat$. 

Monoids with respect to this monoidal structure on the category $\CInvMon$ of commutative inverse monoids are none other than inverse semirings. We conclude that the functor $\Mon(\CInvMon) \to \Mon(\SLat)$ arises from a pseudofunctor which associates idempotent semirings $Q$ to the category of monoids in the functor category $\HomC{\Cat}{Q}{\Ab}$ with respect to Day convolution. And such monoids are simply lax monoidal functors from $Q$ to $\Ab$.

\section{Preliminaries}

In this section we give some background on the categorical concepts necessary for understanding this paper. For more information on 2-categories, see \cite{johnson20212d}.

We start by recalling the notions of monoidal category and monoidal functor.

\begin{definition}
A \emph{monoidal category} is a category $\C$ equipped with a bifunctor
$\otimes \colon \C \times \C \to \C$ (the tensor product), an object $I \in \C$
(the unit), and natural isomorphisms
\[
  \alpha_{X,Y,Z} \colon (X \otimes Y) \otimes Z \to X \otimes (Y \otimes Z),
  \quad
  \lambda_X \colon I \otimes X \to X,
  \quad
  \rho_X \colon X \otimes I \to X,
\]
subject to the pentagon and triangle coherence axioms (see \cite[Chapter~1]{johnson20212d}).

A monoidal category is said to be \emph{strict} if $\alpha$, $\lambda$, and $\rho$
are all identities. A monoidal category $\C$ is \emph{cartesian} if its monoidal
structure is given by finite products.
\end{definition}

\begin{definition}
Let $\C$ and $\D$ be monoidal categories. A \emph{lax monoidal functor}
$F \colon \C \to \D$ consists of a functor $F \colon \C \to \D$ together with a
morphism $\iota \colon I_{\D} \to F I_{\C}$ and a natural transformation
$\mu_{X,Y} \colon FX \otimes FY \to F(X \otimes Y)$, subject to the evident
associativity and unit coherence conditions (see \cite[Chapter 1]{johnson20212d}).

A \emph{strong monoidal functor} is a lax monoidal functor for which $\iota$ and
$\mu$ are isomorphisms. A \emph{strict monoidal functor} between strict monoidal
categories is a strong monoidal functor where $\iota$ and $\mu$ are identities.

A \emph{monoidal natural transformation} between lax monoidal functors
$F, G \colon \C \to \D$ is a natural transformation $\theta \colon F \Rightarrow G$
such that $\theta_{X \otimes Y} \circ \mu^F_{X,Y} = \mu^G_{X,Y} \circ (\theta_X \otimes \theta_Y)$
and $\theta_{I} \circ \iota^F = \iota^G$.
\end{definition}

\begin{definition}
    We write $\MonCat$ for the $2$-category whose objects are monoidal categories, whose $1$-morphisms are strong monoidal functors, and whose $2$-morphisms are monoidal natural transformations. We write $\MonCatlax$ for the similar $2$-category whose $1$-morphisms are lax monoidal functors.
\end{definition}

An \emph{internal monoid} in a monoidal category is given by an object $M$, a multiplication morphism $M\otimes M \ar{m} M$ and a unit morphism $I\ar{e} M$ such that the obvious associativity and unit diagrams commute.

An \emph{internal group} in a cartesian monoidal category is an internal monoid $(G,m,e)$ equipped with an inversion morphism $i\colon G \to G$ such that $m(i, \id_G) = e \circ ! = m(\id_G, i)$. An \emph{internal abelian group} is an internal group satisfying the commutativity condition $m = m (\pi_2,\pi_1)$ where $\pi_{1,2}\colon G \times G \to G$ are the product projections.

Note that the category $\CInvMon$ of commutative inverse monoids has a monoidal structure given by tensor product in the usual way. The unit is the free commutative inverse monoid on one generator $\Z_0$, which is the group $\Z$ with an additional identity adjoined. This monoidal structure restricts to the category $\Ab$ of abelian groups to give the familiar tensor product of $\Z$-modules and to the category $\SLat$ of semilattices.

Now we define some of the notions of fibration we will use in this paper.

\begin{definition}
    Let $P\colon \E \to \C$ be a functor. We say a morphism $g\colon E \to E'$ in $\E$ is \emph{opcartesian} with respect to $P$ if for every $E''\in \E$, every morphism $w\:P(E')\to P(E'')$ in $\C$ and every morphism $e\:E\to E''$ in $\E$ such that $P(e)=w\circ P(g)$, there exists a unique morphism $v\:E'\to E''$ such that $P(v)=w$ and $e = v\circ g$.
	\fibsquniv{E}{E'}{P(E)}{P(E')}{g}{P}{P(g)}{E''}{P(E'')}{e}{\exists! \h v}{w}
    
	A functor $P\: \E \to \C$ is called an \emph{opfibration} if for every object $E\in \E$ and every morphism $f\:P(E)\to C$ in $\C$, there is an opcartesian lift $\ov{f}^E\:E\to f\stb E$ of $f$ to $E$ as in the following diagram.
	\fibsq{E}{f\stb E}{P(E)}{C}{\ov{f}^E}{P}{f}

    A functor $P\: \E \to \C$ is a \emph{fibration} if $P\op\: \E\op \to \C\op$ is an opfibration.
	A \emph{discrete opfibration} $P\:\E\to \C$ is a functor where for every morphism $f\:P(E)\to C$ there is a unique morphism $\ov{f}^E$ such that $P(\ov{f}^E) = f$. Every discrete opfibration is an opfibration. Discrete fibrations are defined similarly.

    A \emph{morphism of (op)fibrations} from $P\: \E \to \B$ to $P'\: \E' \to \B'$ is a pair of functors $G\colon \E \to \E'$ and $F\colon \B \to \B'$ such that the obvious square commutes and $G$ preserves (op)cartesian morphisms. If $P$ and $P'$ are discrete fibrations the last condition is superfluous.

    A 2-morphism of fibrations from $(G_1,F_1)\: P \to P'$ to $(G_2,F_2)\: P \to P'$ is a pair of natural transformations $\tau\: G_1 \to G_2$ and $\sigma\: F_1 \to F_2$ such that $P'\tau = \sigma P$.

    We write $\Fib$, $\Fib[o]$, $\Fib[d]$ and $\Fib[b]$ for the 2-categories of fibrations, opfibrations, discrete fibrations and discrete opfibrations, respectively.
\end{definition}

Opfibrations are in correspondence with \emph{pseudofunctors} into $\Cat$ by the Grothendieck construction. Pseudofunctors are a weak notion of morphism between 2-categories (see \cite{johnson20212d}). We will primarily be concerned with pseudofunctors from 1-categories. 
\begin{definition}
Let $\C$ be a 1-category and let $\D$ be a 2-category. A \emph{pseudofunctor} $F \colon \C \to \D$ consists of an assignment $C \mapsto FC$ on objects, an assignment $f \mapsto Ff$ on morphisms respecting domains and codomains, an invertible 2-morphism
$\iota_C \colon \id_{FC} \to F(\id_C)$ for each object $C$, and an invertible
2-morphism $\mu_{g,f} \colon Fg \circ Ff \to F(g \circ f)$ for each composable pair $f \colon C \to C'$, $g \colon C' \to C''$, subject to associativity and unit coherence axioms.
Note that if the 2-cells $\iota_C$ and $\mu_{g,f}$ are identities the definition reduces to that of a functor.

A \emph{pseudonatural transformation} $\tau \colon F \to G$ between such pseudofunctors consists of a 1-morphism $\tau_C \colon FC \to GC$ for each object $C \in \C$ and an invertible 2-morphism
$\tau_f \: G f \circ \tau_C \to \tau_{C'} \circ Ff$
for each morphism $f \colon C \to C'$ in $\C$ subject to a number of coherence axioms relating $\tau_f$ to the unit and composition 2-cells of $F$ and $G$.

Between pseudonatural transformations there are 3-morphisms called \emph{modifications}. See \cite{johnson20212d} for the formal details.
\end{definition}

\begin{cons}
	Let $\C$ be a category and let $F\:\C\to \Cat$ be a pseudofunctor. We can think of $F$ as a family of categories indexed by $\C$. The Grothendieck construction is a process of reorganisation of this data in terms of a single total category equipped with a projection functor that says which index each object came from. The categories of the family are then recovered by taking the fibres of this projection functor. This process is based on the idea of taking the disjoint union of the categories of the family, but also involves change-of-base operations to handle the relationships between different indices.
	
	The \dfn{Grothendieck construction} of $F$ is the functor $\groth{F}\:\Groth{F}\to \C$ which projects out the indices from the total category $\Groth{F}$, which is defined as follows:
		\begin{itemize}
			\item an object of $\Groth{F}$ is a pair $(C,X)$ with $C\in \C$ and $X\in F(C)$;
			\item a morphism $(C,X)\to (D,X')$ in $\Groth{F}$ is a pair $(f,\alpha)$ with $f\:C\to D$ a morphism in $\C$ and $\alpha\:F(f)(X)\to X'$ a morphism in $F(D)$.
		\end{itemize}
    The functor $\groth{F}$ then sends $(C,X)$ to $C$. It can be shown that this is an opfibration.

    There is also a dual construction sending pseudofunctors $F\colon \C\op \to \Cat$ to \emph{fibrations} into $\C$, for which we will also use the $\Groth{F}$ notation when no confusion is likely.
\end{cons}
For an introduction to the Grothendieck construction for semigroup theorists, see \cite{manuell2022monoid}.

\begin{definition}
    A pseudofunctor $F\colon \C \to \Cat$ from a 1-category $\C$ as above is also called an \emph{opindexed category}. A 1-morphism of opindexed categories from $F\: \C\to\Cat$ to $F'\: \C'\to \Cat$ is a functor $T\: \C \to \C'$ together with a pseudonatural transformation $\tau\: F \to F' T$. A 2-morphism of opindexed categories from $(T_1,\tau_1)\: F \to F'$ to $(T_2,\tau_2)\: F \to F'$ is a natural transformation $\alpha\: T_1 \to T_2$ together with a modification $\aleph\: F' \alpha \circ \tau_1 \to \tau_2$.
    The resulting 2-category is called $\OpICat$.
    A $\C$-indexed category is simply a $\C\op$-opindexed category. We call the 2-category of indexed categories $\ICat$.
    
    An \emph{indexed set} is a functor $F\colon \C\op \to \Set$. By viewing sets as discrete categories, we can view indexed sets as special indexed categories. We call the resulting 2-category of indexed sets $\ISet$.
\end{definition}

The Grothendieck construction extends to a 2-equivalence of 2-categories.
The following fundamental theorem is due to Grothendieck~\cite{grothendieck1971categories} (see also \cite{borceux1994handbook2}).

\begin{theorem}
	The Grothendieck construction extends to a 2-equivalence of 2-categories
	\[\groth{-}\colon \OpICat \aequi \Fib[o].\]
	Given a 1-morphism $(T,\tau) \colon F \to F'$ of opindexed categories, the induced morphism of opfibrations $\groth{(T,\tau)}\colon \Groth{F} \to \Groth{F'}$ lies over $T$ and sends $(C,X)$ to $(TC,\tau_C(X))$ and a morphism $(f,\alpha)\colon (C,X)\to(C',X')$ to $(Tf, \tau_{C'}(\alpha)\circ (\tau_f)_X)$.

	The quasi-inverse sends an opfibration $P\colon \E \to \B$ to the opindexed
	category on $\B$ given by its fibres.
    We obtain a similar 2-equivalence between $\ICat$ and $\Fib$.
	Moreover, the 2-equivalence restricts to one between indexed sets and discrete fibrations.
\end{theorem}

\begin{remark}
    If a functor $P$ is both a fibration and an opfibration, we say it is a \emph{bifibration}. In this case, the change-of-base functor for the fibration is right adjoint to the change-of-base functor for the opfibration.
\end{remark}

In the example in this paper, we have a bifibration where the change of base for the fibration is given by functor composition. Thus, the change-of-base for the opfibration is given by the adjoint to this. This is none other than a Kan extension.

\begin{definition}
    Let $P\colon \A \to \B$ be a functor and let $\C$ be a category. Pre-composition with $P$ induces a functor $P^*\colon \HomC{\Cat}{\B} {\C} \to \HomC{\Cat}{\A}{\C}$ between functor categories.
    If this functor has a left adjoint, we call such an adjoint the \emph{left Kan extension functor along $P$} and denote it by $\Lan{P} \colon \HomC{\Cat}{\A}{\C} \to \HomC{\Cat}{\B}{\C}$.
\end{definition}
In good situations, the left Kan extension is given by the coend
\[(\Lan{P} F)(B) = \int^{A \in \A} \B(P(A), B) \cdot F(A)\]
where $X \cdot F(A)$ denotes the coproduct $\coprod_{x \in X} F(A)$.

\section{Variations on the monoidal Grothendieck construction}\label{secvarmongrconstr}

In this section, we develop some useful variations on the monoidal Grothendieck construction. In particular, we present the discrete monoidal Grothendieck construction as well as variations that involve groups and abelian groups instead of monoids. These, together with the monoidal Grothendieck construction, will constitute key tools for the theory developed in this paper. More precisely, the variations on the discrete monoidal Grothendieck construction will allow us to construct, locally, every Clifford monoid from the corresponding functor into $\Grp$, while the usual monoidal Grothendieck construction, as developed in \cite{MoeVas20}, will allow us to construct, globally, the total category of Clifford monoids.

As presented in \cite{MoeVas20}, the monoidal Grothendieck construction has two faces. This is due to the fact that the operation of fixing the base category in the equivalence of categories given by the Grothendieck construction does not commute with the operation of taking pseudo-monoids. The following diagram illustrates the situation:
\begin{equation*}
\begin{cd}*[3][-4.5]
\& {\Fib\simeq \ICat} \& \\
{\PsMon(\Fib)\simeq \PsMon(\ICat)} \&\& {\Fib_{\X}\simeq \HomC{\TwoCatps}{\X\op}{\Cat}} \\
{\PsMon(\Fib)_{\X}\simeq \HomC{\MonTwoCatps}{\X\op}{\Cat}} \&\& {\PsMon(\Fib_{\X})\simeq \HomC{\TwoCatps}{\X\op}{\MonCat}}
\arrow[mapsto,"{\PsMon(-)}"', from=1-2, to=2-1]
\arrow[mapsto,"{\text{fix base }\X}", from=1-2, to=2-3]
\arrow[mapsto,"{\text{fix base }\X}"', from=2-1, to=3-1]
\arrow[mapsto,"{\PsMon(-)}",from=2-3, to=3-3]
\end{cd}
\end{equation*}
Here $\Fib_\X$ denotes the sub-2-category of $\Fib$ consisting of fibrations with base category $\X$ and 1-morphisms and 2-morphisms which are identities at the base. Also, $\TwoCatps(\X\op,\Cat)$ is the 2-category of pseudofunctors from $\X\op$ to $\Cat$ and $\MonTwoCatps(\X\op,\Cat)$ is the 2-category of lax monoidal pseudofunctors. Finally, $\PsMon$ denotes the operation of taking pseudo-monoids. A pseudo-monoid in a monoidal 2-category is a weakened version of an internal monoid, generalising the notion of monoidal categories in $\Cat$.

In order to develop a discrete version of these two faces of the monoidal Grothendieck construction, we use a similar approach, but this time starting from the equivalence of categories given by the discrete version of the ordinary Grothendieck construction:
$$\Fib[d]\simeq \ISet.$$
For simplicity, and especially since we are primarily interested in the case that $\X$ is a preorder, we will view these as categories instead of 2-categories and hence take monoids instead of pseudo-monoids.
Simultaneously, we also present the variations on this obtained by taking internal groups and internal abelian groups.

\begin{proposition}
    Taking internal monoids in a monoidal category extends to a 2-functor
    $$\Monf(-)\:\MonCat\to \Cat.$$

    Similarly, taking internal groups or internal abelian groups in a cartesian monoidal category extends to 2-functors
    $$\Grpf(-)\:\CartMonCat\to \Cat$$
    $$\Abf(-)\:\CartMonCat\to \Cat.$$
\end{proposition}
\begin{proof}
    It is well-known that taking internal monoids extends to a 2-functor $\Monf(-)$; see, for example, \cite{JanKel01}.
    This is due to the fact that an internal monoid in a monoidal category $\A$ may equivalently be described as a lax monoidal functor from $\1$ to $\A$. Clearly, $\MonCatlax(\1,-)$ yields a 2-functor $\MonCatlax\to \Cat$ that restricts nicely to $\MonCat$. Explicitly, the 2-functor $\Monf(-)$ sends a strong monoidal functor $F\:\A\to \B$ to the functor $\Monf(F)\:\Monf(\A)\to \Monf(\B)$ that maps a monoid $(M,m,e)$ in $\A$ to the monoid in $\B$ given by $F(M)$, multiplication
    $$F(M)\otimes_{\B} F(M)\iso F(M\otimes_{\A} M)\aar{F(m)} F(M)$$
    and unit
    $$I_{\B}\iso F(I_{\A})\aar{F(e)}F(M)$$
    where the unlabelled isomorphisms are given by the structure of strong monoidal functor on $F$. $\Monf(F)$ then sends a morphism of monoids $f$ in $\A$ to $F(f)$, which is automatically a morphism of monoids in $\B$.

    We now prove that taking internal groups in a cartesian monoidal category extends to a 2-functor $\Grpf(-)\:\CartMonCat\to \Cat$. This is due to the fact that for every functor $F\:\A\to \B$ that preserves finite products, where $\A$ and $\B$ are cartesian monoidal categories, the functor $\Monf(F)\:\Monf(\A)\to \Monf(\B)$ restricts to a functor $\Grpf(\A)\to \Grpf(\B)$. Indeed, if $(M,m,e,\invv)$ is a group object in $\A$, we can equip the monoid $F(M)$ described above with the inverse map
    $$F(\invv)\:F(M)\to F(M).$$
    It is straightforward to prove that this yields a 2-functor $\Grpf(-)\:\CartMonCat\to \Cat$. Furthermore, a similar construction yields a 2-functor $\Abf(-)\:\CartMonCat\to \Cat$.
\end{proof}

In the following, we will examine the result of taking internal monoids and fixing the base category in either order and then use the equivalence $\Fib[d]\simeq \ISet$ to obtain our desired results.

\begin{cons}\label{consMonDFib} Let us describe the category $\Monf(\Fib[d])$.
    An object of $\Monf(\Fib[d])$ is a discrete fibration $P\: \E \to \B$ equipped with global multiplication ${\otimes_P} = ({\otimes_\E},{\otimes_\B})$ and unit $I_P = (I_\E,I_\B)$ as shown in the following diagrams:
    \begin{eqD*}
    \begin{cd}*
    {\E \x \E} \& \E \\
	{\B \x \B} \& \B
	\arrow["{\otimes_{\E}}", from=1-1, to=1-2]
	\arrow["{P \x P}"', from=1-1, to=2-1]
	\arrow["P", from=1-2, to=2-2]
	\arrow["{\otimes _{\B}}"', from=2-1, to=2-2]
        \end{cd}
        \qquad \qquad
        \begin{cd}*
           \1 \& \E \\
	\1 \& \B
	\arrow["{I_{\E}}", from=1-1, to=1-2]
	\arrow["{\id{}}"', from=1-1, to=2-1]
	\arrow["P", from=1-2, to=2-2]
	\arrow["{I_{\B}}"', from=2-1, to=2-2]
    \end{cd}
    \end{eqD*}
    These satisfy the internal monoid axioms as shown below:
    \begin{eqD*}
        \begin{cd}*
        \E \& {\E \x \E} \& \E \\
	\B \& {\B \x \B} \& \B
	\arrow["{I_{\E}\x \E}", from=1-1, to=1-2]
	\arrow["{P }"', from=1-1, to=2-1]
	\arrow["{\otimes_{\E}}", from=1-2, to=1-3]
	\arrow["{P\x P}", from=1-2, to=2-2]
	\arrow["P", from=1-3, to=2-3]
	\arrow["{I_{\B}\x \B}"', from=2-1, to=2-2]
	\arrow["{\otimes_{\B}}"', from=2-2, to=2-3]
            \end{cd}
            \qquad = \qquad
            \begin{cd}*
            \E \& \E \\
	\B \& \B
	\arrow[equals, from=1-1, to=1-2]
	\arrow["{P }"', from=1-1, to=2-1]
	\arrow["P", from=1-2, to=2-2]
	\arrow[equals, from=2-1, to=2-2]
            \end{cd}
            \end{eqD*}
            
             \begin{eqD*}
            \begin{cd}*
            {\E \x \E \x \E} \& {\E \x \E} \& \E \\
	{\B \x \B\x \B} \& {\B \x \B} \& \B
	\arrow["{\otimes_{\E} \x \E}", from=1-1, to=1-2]
	\arrow["{P\x P\x P}"', from=1-1, to=2-1]
	\arrow["{\otimes_{\E}}", from=1-2, to=1-3]
	\arrow["{P\x P}", from=1-2, to=2-2]
	\arrow["P", from=1-3, to=2-3]
	\arrow["{\otimes_{\B} \x \B}"', from=2-1, to=2-2]
	\arrow["{\otimes_{\B}}"', from=2-2, to=2-3]
            \end{cd}
            \quad = \quad
            \begin{cd}*
            {\E \x \E \x \E} \& {\E \x \E} \& \E \\
	{\B \x \B\x \B} \& {\B \x \B} \& \B
	\arrow["{\E \x \otimes_{\E} }", from=1-1, to=1-2]
	\arrow["{P\x P\x P}"', from=1-1, to=2-1]
	\arrow["{\otimes_{\E}}", from=1-2, to=1-3]
	\arrow["{P\x P}", from=1-2, to=2-2]
	\arrow["P", from=1-3, to=2-3]
	\arrow["{\B \x \otimes_{\B} }"', from=2-1, to=2-2]
	\arrow["{\otimes_{\B}}"', from=2-2, to=2-3]
            \end{cd}
            \end{eqD*}
    where the other unit law is dual to the shown one.
    
    This corresponds precisely to having strict monoidal structures $(\E, \otimes_{\E}, I_{\E})$ and $(\B, \otimes_{\B}, I_{\B})$ on $\E$ and $\B$ respectively, which make $P$ into a strict monoidal functor. 
    
    A morphism in $\Monf(\Fib[d])$ from $(\E \ar{P} \B, \otimes_{P}, I_P)$ to $(\E' \ar{Q} \B', \otimes_{Q}, I_Q)$ is a pair of functors $(H_0, H_1)$ such that the following square commutes 
    \begin{cd}
        \E \& {\E'} \\
	\B \& {\B'}
	\arrow["{H_1}", from=1-1, to=1-2]
	\arrow["P"', from=1-1, to=2-1]
	\arrow["Q", from=1-2, to=2-2]
	\arrow["{H_0}"', from=2-1, to=2-2]
    \end{cd}
    and the following axioms are satisfied
    \begin{eqD*}
        \begin{cd}*
         {\E \x \E} \& \E \& {\E'} \\
	{\B \x \B} \& \B \& {\B'}
	\arrow["{\otimes_{\E}}", from=1-1, to=1-2]
	\arrow["{P\x P}"', from=1-1, to=2-1]
	\arrow["{H_1}", from=1-2, to=1-3]
	\arrow["P", from=1-2, to=2-2]
	\arrow["Q", from=1-3, to=2-3]
	\arrow["{\otimes_{\B}}"', from=2-1, to=2-2]
	\arrow["{H_0}"', from=2-2, to=2-3]   
        \end{cd}
        \quad = \quad
        \begin{cd}*
        {\E \x \E} \& {\E'\x \E'} \& {\E'} \\
	{\B \x \B} \& {\B' \x \B'} \& {\B'}
	\arrow["{H_1 \x H_1}", from=1-1, to=1-2]
	\arrow["{P\x P}"', from=1-1, to=2-1]
	\arrow["{\otimes_{\E'}}", from=1-2, to=1-3]
	\arrow["{Q\x Q}", from=1-2, to=2-2]
	\arrow["Q", from=1-3, to=2-3]
	\arrow["{H_0 \x H_0}"', from=2-1, to=2-2]
	\arrow["{\otimes_{\B'}}"', from=2-2, to=2-3]
        \end{cd}
    \end{eqD*}
    \begin{eqD*}
        \begin{cd}*
            \1 \& \E \& {\E'} \\
	\1 \& \B \& {\B'}
	\arrow["{I_{\E}}", from=1-1, to=1-2]
	\arrow[equals, from=1-1, to=2-1]
	\arrow["{H_1}", from=1-2, to=1-3]
	\arrow["P"', from=1-2, to=2-2]
	\arrow["Q", from=1-3, to=2-3]
	\arrow["{I_{\B}}"', from=2-1, to=2-2]
	\arrow["{H_0}"', from=2-2, to=2-3]
        \end{cd}
        \quad = \quad
        \begin{cd}*
            \1 \& {\E'} \\
	\1 \& {\B'}
	\arrow["{I_{\E'}}", from=1-1, to=1-2]
	\arrow[equals, from=1-1, to=2-1]
	\arrow["Q", from=1-2, to=2-2]
	\arrow["{I_{\B'}}"', from=2-1, to=2-2]
        \end{cd}
    \end{eqD*}
    These correspond precisely to morphisms of discrete fibrations $(H_0,H_1)$ (i.e.\ just morphisms in the arrow category of $\Cat$) where $H_0$ and $H_1$ are strict monoidal functors.
\end{cons}

We can now restrict to monoidal discrete fibrations with given base.
\begin{cons}\label{consMonDFibX}
Given a strict monoidal category $(\X, \otimes_{\X}, I_{\X})$, an object of ${\Monf(\Fib[d])_{\X}}$ is a discrete fibration $P\: \E \to \X$ from a strict monoidal category $(\E, \otimes_{\E}, I_{\E})$ to $(\X, \otimes_{\X}, I_{\X})$ that is a strict monoidal functor. A morphism in ${\Monf(\Fib[d])_{\X}}$ from $P\: \E \to \X$ to $Q\: \E' \to \X$ is a strict monoidal functor $H\: \E \to \E'$ such that the following triangle commutes. 
\begin{cd}
    \E \&[-3ex] \&[-3ex] {\E'} \\
	\& \X
	\arrow["H", from=1-1, to=1-3]
	\arrow["P"', from=1-1, to=2-2]
	\arrow["Q", from=1-3, to=2-2]
\end{cd}
\end{cons}

Alternatively, if we first restrict to discrete fibrations with a given base before taking monoid objects, we obtain the following.

\begin{cons}\label{consMonDFibXins} Let $\X$ be a category. We consider the category $\Monf(\Fib[d]_{\X})$.
An object of $\Monf(\Fib[d]_{\X})$ is a discrete fibration $P:\E \to \X$ equipped with fibrewise multiplication $\otimes\: \E \x[\X] \E \to \E$ and unit $I\colon \X \to \E$ as in the diagrams below:
\begin{eqD*}
    \begin{cd}*
    {\E \x[\X] \E} \& \E \\
	{\X} \& \X
	\arrow["{\otimes}", from=1-1, to=1-2]
	\arrow["{P \boxtimes P}"', from=1-1, to=2-1]
	\arrow["P", from=1-2, to=2-2]
	\arrow[equal, from=2-1, to=2-2]
        \end{cd}
        \qquad \qquad
        \begin{cd}*
           \X \& \E \\
	\X \& \X
	\arrow["{I}", from=1-1, to=1-2]
	\arrow[equal, from=1-1, to=2-1]
	\arrow["P", from=1-2, to=2-2]
	\arrow[equal, from=2-1, to=2-2]
        \end{cd}
        \end{eqD*}
    where $P\boxtimes P$ is the composite map from $\E\x[\X] \E\to \X$ given by either of the two paths in the pullback square of $P$ along itself. 
        
    These satisfy identity and associativity axioms similar to the ones of \consx\ref{consMonDFib}. In particular, given $X\in \X$, we have a monoid structure $(P_X, \otimes_X, I_X)$ on the fibre $P_X$ of $P$ over $X$. Moreover, the reindexing functors of $P$ are automatically monoid homomorphisms. This is what will make the usual Grothendieck construction of $P$ factor through $\Mon$.

    A morphism in $\Monf(\Fib[d]_{\X})$ from $(\E \ar{P} \X, \otimes_P,I_P)$ to  $(\E' \ar{Q} \X, \otimes_Q,I_Q)$  is a functor $H\: \E \to \E'$ such that 
    \begin{cd}
        {\E} \& {\E'} \\
    \X \& \X
	\arrow["H", from=1-1, to=1-2]
	\arrow["P"', from=1-1, to=2-1]
	\arrow["Q", from=1-2, to=2-2]
	\arrow[equals, from=2-1, to=2-2]
        \end{cd}
    that satisfies axioms of preservation of unit and composition similar to the ones of \consx \ref{consMonDFib}. In particular, this ensures that for every $X\in \X$ we have that $H_X\: P_X \to Q_X$ is a monoid homomorphism.
\end{cons}

Now we take monoids on the indexed-category side of the equivalence.

\begin{cons}\label{consMonISet}
We consider the category $\Monf(\ISet)$.
An object of $\Monf(\ISet)$ is an indexed set $(\C, \C\op \ar{F} \Set)$ equipped with a multiplication $\otimes_F = (\otimes_\C,\tau)$ of the form 
$$(\C \x \C, \C\op \x \C \op \ar{F\x F} \Set \x \Set) \ar{(\otimes_{\C}, \tau)} (\C, \C \op \ar{F} \Set)$$
and a unit $I_F = (I_\C,\eta)$ of the form
$$(\1 , \1\op \ar{\Delta 1} \Set) \ar{(I_{\C} , \eta)} (\C , \C \op \ar{F} \Set),$$
where $\tau$ and $\eta$ are natural transformations as shown below.
\begin{eqD*}
\begin{cd}*[4]
    {\C\op \x \C\op } \& {\Set \x \Set} \& \Set \\
	{\C \op}
	\arrow["{F \x F}", from=1-1, to=1-2]
	\arrow[from=1-1, to=2-1,"{\otimes_{\C}}"']
	\arrow["\tau", shift left=22, Rightarrow, from=1-1, to=2-1, shorten <= 1ex, shorten >= 0ex ]
	\arrow["\x", from=1-2, to=1-3]
	\arrow["F"', bend right=20, from=2-1, to=1-3]
\end{cd}
\qquad 
\begin{cd}*[4.5]
{\1\op} \& \Set \\
	{\C \op}
	\arrow["\Delta1", from=1-1, to=1-2]
	\arrow["I^{\operatorname{op}}_{\C}"', from=1-1, to=2-1]
	\arrow["\eta", shift left=9, Rightarrow, from=1-1, to=2-1, shorten <= 0.8ex, shorten >= 1ex]
	\arrow["F"', bend right= 25, from=2-1, to=1-2] 
\end{cd}
\end{eqD*}
This data satisfies the corresponding axioms of unitality and associativity. This corresponds precisely to a strict monoidal structure $(\C, \otimes_{\C}, I_{\C})$ on $\C$ together with a lax monoidal structure on $F\: \C \op \to \Set$ given by $\tau$ and $\eta$.

A morphism in $\Monf(\ISet)$ from $((\C, \C\op \ar{F} \Set), \otimes_F, I_F)$ to $((\D, \D\op \ar{G} \Set), \otimes_G,I_G)$ is a pair $(H, \phi)$ where $H\:\C \to \D$ is a functor and $\phi$ is a natural transformation 
\begin{cd}
    {\C\op} \& \Set \\
	{\D \op}
	\arrow["F", from=1-1, to=1-2]
	\arrow["{H\op}"', from=1-1, to=2-1]
	\arrow["\phi", shift left=7, Rightarrow, from=1-1, to=2-1, shorten <= 1ex, shorten >= 1ex]
	\arrow["G"', bend right=25, from=2-1, to=1-2]
\end{cd}
such that the axioms of a morphism of internal monoids are satisfied. This corresponds precisely to $H$ being a strict monoidal functor and $\phi$ being a monoidal natural transformation.
\end{cons}

Again fixing the base we have the following.

\begin{cons}\label{consMonISetX}
Given a strict monoidal category $(\X, \otimes_{\X}, I_{\X})$, an object of $\Monf(\ISet)_{\X}$ is an object of $\Monf(\ISet)$ that has base $\X$. That is, a lax monoidal functor $\X\op \ar{F} \Set$. A morphism in $\Monf(\ISet)_{\X}$ from $\X\op \ar{F} \Set$ to $\X\op \ar{G} \Set$ is then simply a monoidal natural transformation $\alpha\: F \Rightarrow G$. So the category $\Monf(\ISet)_{\X}$ is isomorphic to the category $\HomC{\MonCatlax}{\X\op}{\Set}$.
\end{cons}

Finally, if we first fix the domain of the indexed set and then take monoids we arrive at the following construction.

\begin{cons}\label{consMonXopSet}
Let $\X$ be a category. It is well-known that, since $\Cat(\X\op,-)$ preserves limits, we have an isomorphism between $\Monf(\HomC{\Cat}{\X\op}{\Set})$ and the category $\HomC{\Cat}{\X\op}{\Mon}$.
\end{cons}

We can now state the main theorem of this section.

\begin{theorem}[The discrete monoidal Grothendieck construction]\label{theordiscmongroth}
    There are equivalences of categories
    \begin{cd}[3][-4.5]
    \& {\Fib[d]\simeq \ISet} \& \\
	{\Monf(\Fib[d])\simeq \Monf(\ISet)} \&\& {\Fib[d]_{\X}\simeq \HomC{\Cat}{\X\op}{\Set}} \\
	{\Monf(\Fib[d])_{\X}\simeq \HomC{\MonCatlax}{\X\op}{\Set}} \&\& {\Monf(\Fib[d]_{\X})\simeq \HomC{\Cat}{\X\op}{\Mon}}
	\arrow[mapsto,"{\Monf(-)}"', from=1-2, to=2-1]
	\arrow[mapsto,"{\text{fix base }\X}", from=1-2, to=2-3]
	\arrow[mapsto,"{\text{fix base }\X}"', from=2-1, to=3-1]
	\arrow[mapsto,"{\Monf(-)}",from=2-3, to=3-3]
\end{cd}
\end{theorem}
\begin{proof}
    The equivalence of categories $\Fib[d]\simeq \ISet$ is an internal equivalence in $\MonCat$. Indeed, both $\Fib[d]$ and $\ISet$ are cartesian monoidal categories. And since the two pseudo-inverses form an equivalence, they both preserve limits and are thus strong monoidal functors.
    It is straightforward to show that the natural isomorphisms regulating the equivalence of categories are monoidal, thus exhibiting an internal equivalence in $\MonCat$. Moreover, it is also straightforward to show that restricting to a fixed base $\X$ preserves the equivalences that we have. This is due to the fact that the equivalence $\Fib[d]\simeq \ISet$ is over $\Cat$, which means that it forms a commutative triangle with the codomain functor $\Fib[d]\to \Cat$ and with the projection functor $\pi_1\:\ISet\to \Cat$ that takes the base category. Since any 2-functor sends equivalences to equivalences, in particular $\Monf(-)$ does so. We then conclude by applying the results of the constructions above.
\end{proof}

\begin{cons}\label{consexplGrMon}
    We can extract from \thex\ref{theordiscmongroth} the explicit action of the discrete monoidal Grothendieck constructions. On the left leg, consider a lax monoidal functor $F\: \X\op \to \Set$, presented by the natural transformations
    \begin{eqD*}
        \begin{cd}*
         {\X\op \x \X \op} \&\& {\X \op} \\
	{\Set \x \Set} \&\& \Set
	\arrow["{\otimes_{\X}^{\operatorname{op}}}", from=1-1, to=1-3]
	\arrow["{F \x F}"', from=1-1, to=2-1]
	\arrow["\tau", shorten <= 4ex, shorten >= 4ex, Rightarrow, from=2-1, to=1-3]
	\arrow["F", from=1-3, to=2-3]
	\arrow["\x"', from=2-1, to=2-3]   
        \end{cd}
        \qquad\qquad
        \begin{cd}*
            \1 \&\& {\X \op} \\
	\&\& \Set
	\arrow["{I^{\operatorname{op}}_{\X}}", from=1-1, to=1-3]
	\arrow["1"',""{name=0, anchor=center, inner sep=0}, from=1-1, to=2-3]
	\arrow["F", from=1-3, to=2-3]
	\arrow["\eta"', shorten <= 1ex, shorten >= 1ex, Rightarrow, from=0, to=1-3]
        \end{cd}
    \end{eqD*}
    The corresponding discrete fibration $\Groth{F}\to \X$ is constructed as in the usual Grothendieck construction. The total category $\Groth{F}$ is then equipped with a global monoidal structure with monoidal unit 
    $$I_{\Groth{F}}=(I_{\X},\eta)$$
    and a tensor product, which is defined on objects by 
    $$ (X, C) \otimes_{\Groth{F}} (X',D) =(X\otimes_{\X} X',\tau_{X,X'}(C,D))$$
    and on morphisms by
    $$f\otimes_{\Groth{F}} g = f\otimes_{\X} g.$$
    
    On the right leg, given a presheaf of monoids $F\: \X\op \to \Mon$ the corresponding discrete fibration $\Groth{F}\to \X$ is the usual Grothendieck construction applied to the presheaf 
    $$\X\op \ar{F} \Mon \ar{U} \Set,$$ where $U$ is the forgetful functor. Notice that the fibres are automatically equipped with a monoid structure since $F$ takes values in $\Mon$.
\end{cons}

Note that the non-commutativity of the operations of taking monoids and fixing the base observed in \cite{MoeVas20} is already on display in the discrete version of the monoidal Grothendieck construction. There it is shown that the constructions coincide when the base category $\X$ is (strictly) cartesian monoidal. We now show this directly in the discrete case.

\begin{proposition}\label{propcartmon}
Let $\X$ be a strict cartesian monoidal category. There is an equivalence of categories
$${\Monf(\Fib[d])_{\X}\simeq \HomC{\MonCatlax}{\X\op}{\Set}} \simeq{\Monf(\Fib[d]_{\X})\simeq \HomC{\Cat}{\X\op}{\Mon}}$$
\end{proposition}
\begin{proof}[Proof (sketch)]
    Let $F\: \X\op \to \Set$ be a lax monoidal functor. Then $F$ corresponds to a monoid in $\Fib[d]$ with base $\X$. Call this discrete fibration $P\:\E\to \X$. The global strict monoidal structure on $\E$ is given by the natural transformations $\tau$ and $\eta$ that make $F$ into a lax monoidal functor, as in \consx \ref{consMonISet}. We induce a local monoid structure on each fibre $F(X)$ as follows:
    $$\otimes_X\: F(X) \x F(X) \ar{\otimes_\E=\tau_{X,X}} F(X \x X) \ar{F((\id{}, \id{}) )} F(X)$$
    $$I_X\: 1 \ar{I_\E=\eta} F(1) \ar{F(!)} F(X)$$ 
    It is then straightforward to check that $(F(X), \otimes_X, I_X)$ is a monoid and $F(f)$ is automatically a monoid homomorphism for every morphism $f$ in $\X$.

    On the other hand, given a functor $F\: \X\op \to \Mon$, consider the local monoid structures $(F(X),\otimes_X,I_X)$ on the fibres. We construct the corresponding global strict monoidal structure $\otimes_{\Groth{F}}$ on $\Groth{F}$ as follows. Define
    $$\tau_{X,Y}=( F(X)\x F(Y) \ar{F(\pi_1)\x F(\pi_2)} F(X\x Y)\x F(X\x Y) \ar{\otimes_{X\x Y}} F(X\x Y))$$
    $$\eta=I_1\: 1 \to F(1)$$
    We construct
    $$(X,a)\otimes_{\Groth{F}}(Y,b)=(X\times Y, \tau_{X,Y}(a,b))$$
    $$I_{\Groth{F}}=(1,\eta).$$
    It is then straightforward to check that this gives a global monoidal structure on $\Groth{F}$ making $\Groth{F}\to \X$ into a monoid in $\Fib[d]$. By the discrete monoidal Grothendieck construction, this is equivalent to a lax monoidal functor $\X\op \to \Set$.
\end{proof}

Similarly to how we obtained the discrete monoidal Grothendieck construction, we can also apply the 2-functors $\Grpf$ and $\Abf$ instead of $\Monf$ to the equivalence of categories $\Fib[d]\simeq \ISet$ and obtain the following result.

\begin{theorem}\label{theorgrothconstrgrpab}
    There are equivalences of categories
    \begin{cd}[3][-4.5]
    \& {\Fib[d]\simeq \ISet} \& \\
	{\Grpf(\Fib[d])\simeq \Grpf(\ISet)} \&\& {\Fib[d]_{\X}\simeq \HomC{\Cat}{\X\op}{\Set}} \\
	{\Grpf(\Fib[d])_{\X}\simeq \Grpf(\ISet)_{\X}} \&\& {\Grpf(\Fib[d]_{\X})\simeq \HomC{\Cat}{\X\op}{\Grp}}
	\arrow[mapsto,"{\Grpf(-)}"', from=1-2, to=2-1]
	\arrow[mapsto,"{\text{fix base }\X}", from=1-2, to=2-3]
	\arrow[mapsto,"{\text{fix base }\X}"', from=2-1, to=3-1]
	\arrow[mapsto,"{\Grpf(-)}",from=2-3, to=3-3]
\end{cd}\v[1]
\begin{cd}[3][-4.5]
    \& {\Fib[d]\simeq \ISet} \& \\
	{\Abf(\Fib[d])\simeq \Abf(\ISet)} \&\& {\Fib[d]_{\X}\simeq \HomC{\Cat}{\X\op}{\Set}} \\
	{\Abf(\Fib[d])_{\X}\simeq \Abf(\ISet)_{\X}} \&\& {\Abf(\Fib[d]_{\X})\simeq \HomC{\Cat}{\X\op}{\Ab}}
	\arrow[mapsto,"{\Abf(-)}"', from=1-2, to=2-1]
	\arrow[mapsto,"{\text{fix base }\X}", from=1-2, to=2-3]
	\arrow[mapsto,"{\text{fix base }\X}"', from=2-1, to=3-1]
	\arrow[mapsto,"{\Abf(-)}",from=2-3, to=3-3]
\end{cd}
\end{theorem}

\begin{remark}\label{remexpl}
    We can extract from \thex\ref{theorgrothconstrgrpab} the explicit Grothendieck constructions for the groups and abelian groups variations of the discrete monoidal Grothendieck construction. These are completely analogous to the ones described in \consx \ref{consexplGrMon} in the monoidal case. 
\end{remark}

In \cref{sec:monoids_in_monoidal_fibrations,secinvsemirings} we will also make use of the non-discrete monoidal Grothendieck construction developed by Moeller and Vasilakopoulou. The details can be found in \cite{MoeVas20}, but conceptually, it is simply a 2-dimensional upgrading of the discrete monoidal Grothendieck construction described above. We just recall here below the definition of monoidal fibration in the global and in the fibrewise case. Analogous definitions can be given for opfibrations.

\begin{definition}[\cite{MoeVas20}]\label{defmonoidalfibration}
    A \emph{monoidal fibration} is a pseudomonoid in $(\Fib,\times,1_\1)$. Equivalently, it is a fibration for which both the total and base category are monoidal, the underlying functor is a strict monoidal functor and the tensor product of the total category preserves cartesian liftings.
\end{definition}

Notice that the previous definition is the non-discrete version of Construction \ref{consMonDFib}. That is, the case in which we have a global tensor product in the total category. The tensor product of objects $C$ over $X$ and $D$ over $X'$ is an object in the total category over the tensor product $X\otimes X'$ in the base.

We now consider the fibrewise version of the notion of monoidal fibration that will be the non-discrete version of Construction \ref{consMonDFibX}. We will call this, which has no name in \cite{MoeVas20}, a \emph{fibrewise monoidal fibration}.

\begin{definition}[\cite{MoeVas20}]
    A \emph{fibrewise monoidal fibration} is a pseudomonoid in $(\Fib_\X,\boxtimes,1_\X)$. Explicitly, it is an ordinary fibration $P\colon\A\to\X$ equipped with two fibred functors $(m,1_\X)\colon P\boxtimes P\to P$ and $(j,1_\X)\colon 1_\X\to P$ 
along with invertible fibred 2-cells satisfying the usual axioms. 
\end{definition}

Note that, given a fibrewise monoidal fibration, we can only take the tensor product of two objects in the total category that belong to the same fibre.

\cite{MoeVas20} also describes explicit constructions for the two legs of the non-discrete monoidal Grothendieck construction; those constructions are similar to \consx\ref{consexplGrMon}.

\begin{remark}\label{remnotationupg}
    Following the notation of \cite{MoeVas20}, we will denote as $\MonTwoCatps$ the tricategory of monoidal 2-categories, weakly lax monoidal pseudofunctors, weakly monoidal pseudonatural transformations and monoidal modifications. The fact that this forms a tricategory follows from the result in \cite{CheGur11} that $\catfont{MonBicat}$ is a tricategory.
    For every pair $(\A,\B)$ of monoidal 2-categories, the hom-object $\HomC{\MonTwoCatps}{\A}{\B}$ is a strict 2-category. Moreover, restricting to weakly lax monoidal 2-functors and weakly monoidal 2-natural transformations, we obtain a sub-2-category $\HomC{\MonTwoCat}{\A}{\B}$. As proved in \cite{MoeVas20}, taking pseudo-monoids in a monoidal 2-category then extends to a 2-functor
    $$\PsMon(-)\simeq\HomC{\MonTwoCat}{\1}{-}\:\MonTwoCatps\to \twoCat$$
    which maps a monoidal 2-category to the 2-category of internal pseudo-monoids in it, strong morphisms and 2-cells between them.
\end{remark}

\section{Clifford monoids as discrete fibrations over semilattices}

We are now in a position to explain the structure theorem for Clifford monoids in terms of the monoidal Grothendieck construction.

Recall from \cref{sec:introduction} that a Clifford monoid can be described in terms of a functor $G\colon L\op \to \Grp$ where $L$ is a meet-semilattice and that the original monoid can be recovered as $\{(e,g) \mid e \in L, g \in G(e)\}$
with \[(e,g) \cdot (e',g') = (ee', G(ee' \le e)(g) \cdot G(ee' \le e')(g')).\]

Since meet-semilattices are essentially the same as cartesian monoidal posets, we can apply the construction from \cref{propcartmon} to $G$ to obtain a monoidal discrete fibration into $L$.
We find that the objects of $\Groth{G}$ and the action of the tensor product on them are precisely the elements of the reconstructed monoid above. Moreover, the order structure is such that $(e,g) \le (e',g')$ if and only if $e \le e'$ and $G(e \le e')(g') = g$. This holds when $(e,g) = (e',g') \cdot (e,1)$ and so we actually recover the usual order on a Clifford monoid. The functor $\Groth{G} \to L$ simply finds the idempotent associated to each element of the Clifford monoid.

In fact, we can go further and use our characterisations from \cref{secvarmongrconstr} to give a new conceptual proof of the structure theorem. Let us start by proving a few useful lemmas.

\begin{lemma}\label{lem:clifford_gives_fibration}
  Let $M$ be a Clifford monoid and let $\eta_M\colon M \to E(M)$ be the monoid homomorphism sending $x$ to $e_x = xx^*$. (This is the universal semilattice quotient, see \cref{prop:both_adjoints}). If $M$ is viewed as a monoidal poset with its canonical order, then $\eta_M$ is a monoidal discrete fibration.
\end{lemma}
\begin{proof}
  Since $\eta_M$ is a monoid homomorphism, it is a strict monoidal functor. To see it is a discrete fibration, suppose $e \le e_m = \eta_M(m)$ for $e \in E(M)$ and $m \in M$. Then a lift of this is an inequality $a \le m$ such that $e_a = e$. Note that $a \le m$ means $a = me$ and so this is the only possible lift. Finally, note that $\eta_M(me) = e_m \wedge e = e$ and so this is indeed a lift and $\eta_M$ is a discrete fibration.
\end{proof}
 
\begin{lemma}\label{lemmapos}
    Let $\X$ be a poset and let $\E \ar{P} \X$ be a discrete fibration. Then the total category $\E$ is a poset.
\end{lemma}

\begin{proof}
Given two morphisms $f,g\: x\to y$ in $\E$, we have $P(f)=(P(x)\leq P(y))=P(g)$ in $\X$, since $\X$ is a poset. But both $f$ and $g$ are lifts of $P(x) \le P(y)$ and so $f=g$ by the uniqueness of lifts. Thus, $\E$ is a preorder.
Furthermore, if $x \cong y$ then $P(x) = P(y)$ and so by uniqueness of the lift of the identity on $P(y)$, we have $x = y$ and so $\E$ is a poset.
\end{proof}

It will be convenient to consider subcategories of $\Cliff$ with a given semilattice of idempotents.
\begin{definition} 
    Let $L$ be a semilattice. Then $\Cliff_L$ is the subcategory consisting of those Clifford monoids whose semilattice of idempotents is equal to $L$ and those monoid homomorphisms that act as the identity on these idempotents.
\end{definition}

\begin{theorem}\label{thm:group_fibration_over_semilattice}
    Let $L$ be a semilattice. There is an equivalence of categories 
    \[\Cliff_L\simeq \Grpf(\Fib[d]_{L}).\]
\end{theorem}
\begin{proof}
    Note that by \cref{propcartmon} and \cref{lemmapos}, an object of $\Monf(\Fib[d]_{L})$ is a monoid equipped with a homomorphism into $L$ and ordered so that this map is a discrete fibration. An object of $\Grpf(\Fib[d]_{L})$ has the additional property that the fibres are groups.

    By \cref{lem:clifford_gives_fibration}, every Clifford monoid in $\Cliff_L$ gives such an object of $\Grpf(\Fib[d]_{L})$.

    Conversely, let $h\: M \to L$ be an object of $\Grpf(\Fib[d]_{L})$. Then the group inverses in the fibres define an involution $(-)^*\colon M \to M$ such that $x x^* x = x$ for all $x \in M$. Furthermore, the idempotents $x x^*$ correspond to the unit of the fibres and hence commute with all elements by the commutativity of $L$ and the global multiplication from \cref{propcartmon}. Hence, $M$ is a Clifford monoid.
    
    Moreover, as there is exactly one unit in each fibre, they map bijectively onto $L$, and they are the only idempotents since the fibres are groups.
    Note that by transporting the elements of $L$ along the bijection, we obtain an isomorphic copy of $M$ in which the idempotents are exactly $L$ (and also replacing the other elements if necessary to avoid clashes).
    
    As in the proof of \cref{lem:clifford_gives_fibration} the order on the domain is forced to be the usual order on a Clifford monoid. Thus, up to isomorphism, every object in $\Grpf(\Fib[d]_{L})$ arises from a Clifford monoid.

    Finally, every morphism $f\colon M \to M'$ in $\Cliff_L$ commutes with the maps to $L$ since homomorphisms preserve associated idempotents and the idempotents are fixed by definition. Thus, they coincide with the morphisms in $\Grpf(\Fib[d]_{L})$.
\end{proof}

\begin{remark}
     Despite having that $L$ is a cartesian monoidal poset, we cannot rewrite $\Grpf(\Fib[d]_L)$ as $\Grpf(\Fib[d])_L$, because $L$ is just a strict monoidal category and has in general no structure of group object in $\Cat$. Indeed, Clifford monoids only admit generalised inverses and not actual inverses with respect to the monoid structure.
\end{remark}

The desired result now follows from the monoidal Grothendieck construction.

\begin{corollary}\label{cor:clifford_functor}
    For any semilattice $L$, we have an equivalence of categories
    \[\Cliff_L \simeq \HomC{\Cat}{L\op}{\Grp}.\]
\end{corollary}
\begin{proof}
    Simply apply the variant of the Grothendieck construction from \cref{theorgrothconstrgrpab}.
    The resulting equivalence is precisely the known correspondence between Clifford monoids with semilattice of idempotents $L$ and contravariant functors from $L$ into $\Grp$.
    On morphisms, the correspondence sends a morphism $f\colon M \to N$ in $\Cliff_L$ to a natural transformation $\phi$ defined by $\phi_e(g) = f(g)$.
\end{proof}

Note that we now can also restrict the equivalences of \cref{thm:group_fibration_over_semilattice,cor:clifford_functor} to the commutative case.

\begin{corollary}\label{cor:commutative_clifford_functor}
    For a semilattice $L$, there is an equivalence of categories
    \[\CInvMon_L\simeq \Abf(\Fib[d]_{L})\simeq \HomC{\Cat}{L\op}{\Ab}.\]
\end{corollary}

\begin{remark}
    The construction of a Clifford monoid from a functor $G\colon L\op \to \Grp$ is generalised by the notion of a Płonka sum (see \cite{plonka1967method}). These also correspond to various notions of algebraic Grothendieck construction with the roles of groups or monoids replaced by different algebraic structures.
\end{remark} 

\section{The total category of Clifford monoids}

In the previous section we studied equivalences between Clifford monoids and functors, though in each case we kept the semilattice of idempotents fixed. In this section we remedy this and use the Grothendieck construction to combine each of these categories into one, which we then show is equivalent to the category of Clifford monoids.

Then, restricting ourselves to commutative Clifford monoids, we obtain the structure of a monoidal opfibration. This section thus globally captures (commutative) Clifford monoids via the (monoidal) Grothendieck construction, while the previous section locally captured each (commutative) Clifford monoid via variations of the discrete monoidal Grothendieck construction.   
We start by proving some basic properties of the functor which takes idempotents.

\begin{proposition}\label{prop:both_adjoints}
    The functor $E\colon \Cliff \to \SLat$ is both left and right adjoint to the inclusion from $\SLat$ to $\Cliff$.
\end{proposition}
\begin{proof}
  To show $E$ is a right adjoint we set $\epsilon_M\colon E(M) \to M$ to be the obvious inclusion of idempotents and note that monoid homomorphisms preserve idempotents and thus if $S$ is a semilattice, the map $f\colon S \to M$ factors (uniquely) through $E(M)$ as in the following diagram.
  \[\begin{tikzcd}
	S & M \\
	& {E(M)}
	\arrow["f", from=1-1, to=1-2]
	\arrow[dashed, from=1-1, to=2-2]
	\arrow["{\epsilon_M}"', hook, from=2-2, to=1-2]
  \end{tikzcd}\]
  To see $E$ is a left adjoint we first define $\eta_M\colon M \to E(M)$ by $\eta_M(x) = e_x = xx^*$. This is a monoid homomorphism since $1 = 1 \cdot 1^*$ and $(xy)(xy)^* = xy y^* x^* = x e_y x^* = xx^* e_y = e_x e_y$. Now consider $f\colon M \to S$ where $S$ is a semilattice. We have $f(e_x) = f(xx^*) = f(x)f(x)^* = f(x)$ and so $f$ factors (uniquely) through  $\eta_M$ as in the diagram below.
  \[\begin{tikzcd}
	M & S \\
	{E(M)}
	\arrow["f", from=1-1, to=1-2]
	\arrow["{\eta_M}"', two heads, from=1-1, to=2-1]
	\arrow[dashed, from=2-1, to=1-2]
  \end{tikzcd}\]
  Thus, we have proved the result.
\end{proof}

\begin{corollary}\label{cor:E_fibration}
    The functor $E$ is both a fibration and an opfibration. Explicitly, a cartesian lift of $f\colon L \to E(B)$ is given by the pullback
    \[\begin{tikzcd}
	{L \times_{E(B)} B} & B \\
	L & {E(B)}
	\arrow[from=1-1, to=1-2]
	\arrow[from=1-1, to=2-1]
	\arrow["\lrcorner"{anchor=center, pos=0.125}, draw=none, from=1-1, to=2-2]
	\arrow["{{\eta_B}}", from=1-2, to=2-2]
	\arrow["f"', from=2-1, to=2-2]
    \end{tikzcd}\]
    and an opcartesian lift of $g\colon E(A) \to L$ is given by the pushout
    \[\begin{tikzcd}
	A & {A +_{E(A)} L} \\
	{E(A)} & L \, .
	\arrow[from=1-1, to=1-2]
	\arrow["\lrcorner"{anchor=center, pos=0.125, rotate=-90}, draw=none, from=1-2, to=2-1]
	\arrow["{{\epsilon_A}}", from=2-1, to=1-1]
	\arrow["g"', from=2-1, to=2-2]
	\arrow[from=2-2, to=1-2]
\end{tikzcd}\]
\end{corollary}
\begin{proof}
    Since $E$ is a right adjoint, it preserves pullbacks, and a reflection that preserves pullbacks is a Street fibration. (By making an appropriate choice of pullback, we can ensure this is actually a Grothendieck fibration.)
    Dually, $E$ is also an opfibration.
\end{proof}

\begin{remark}
    Note that the pushout required to construct the opcartesian lift can be obtained as a quotient of the product $A \times L$, since the elements of $L$ are idempotent and hence central in the coproduct Clifford monoid $A + L$. Consequently, when restricting to the category of commutative Clifford monoids, the opcartesian lifts remain the same.
\end{remark}

As an aside, we now give a nice semigroup-theoretic way to describe the cartesian morphisms for this fibration. 
\begin{lemma}\label{lem:cartesian_characterisation}
    A morphism $f\colon M \to N$ in $\Cliff$ is cartesian with respect to $E\colon \Cliff \to \SLat$ if and only if it restricts to a group isomorphism $G_e \to G_{f(e)}$ for each idempotent $e \in E(M)$.
\end{lemma}
\begin{proof}
Recall that $f\colon M \to N$ is cartesian if and only if the naturality square
\[\begin{tikzcd}
	M & N \\
	{E(M)} & {E(N)}
	\arrow["f", from=1-1, to=1-2]
	\arrow["{\eta_M}"', two heads, from=1-1, to=2-1]
	\arrow["{\eta_N}", two heads, from=1-2, to=2-2]
	\arrow["{E(f)}"', from=2-1, to=2-2]
\end{tikzcd}\]
is a pullback. We can also describe this condition more concretely. Consider the induced map $t$ from $M$ into the canonical pullback monoid $E(M) \times_{E(N)} N = \{ (e,n) \in E(M) \times N \mid f(e) = e_n \}$ as in the following diagram.
\[\begin{tikzcd}
	M && \\
	& {E(M) \times_{E(N)} N} & N \\
	& {E(M)} & {E(N)}
	\arrow["t", dashed, from=1-1, to=2-2]
	\arrow["f", curve={height=-12pt}, from=1-1, to=2-3]
	\arrow["{\eta_M}"', curve={height=12pt}, from=1-1, to=3-2]
	\arrow[from=2-2, to=2-3]
	\arrow[from=2-2, to=3-2]
	\arrow["{{\eta_N}}", two heads, from=2-3, to=3-3]
	\arrow["{{E(f)}}"', from=3-2, to=3-3]
\end{tikzcd}\]
We have that $t(m) = (e_m,f(m))$ should be an isomorphism. Consider $m, m'$ with the same associated idempotent. Injectivity of $t$ means that $f(m) = f(m')$ implies $m = m'$. Moreover, take $e \in E(M)$ and consider $n \in N$ with idempotent $f(e)$. Then surjectivity of $t$ means that there is an $m$ with idempotent $e$ such that $f(m) = n$. In other words, $f$ is cartesian if and only if it restricts to a bijection $G_e \to G_{f(e)}$ for each $e \in E(M)$.
\end{proof}

\begin{theorem}\label{thm:main_equivalence}
    Consider the functor $\HomC{\Cat}{(-)\op}{\Grp}\colon \SLat\op \to \Cat$. There is an equivalence of categories 
    \[\Groth{\HomC{\Cat}{(-)\op}{\Grp}} \simeq \Cliff.\]
\end{theorem}

\begin{proof}
    It suffices to show that the inverse of the Grothendieck construction applied to the fibration $E\:\Cliff\to\SLat$ yields a pseudofunctor which is pseudonaturally isomorphic to $\HomC{\Cat}{(-)\op}{\Grp}$.

   Let us write $\Cliff_{(-)}\colon \SLat\op \to \Cat$ for the pseudofunctor associated to $E$.
   On objects this agrees with our previous definition of $\Cliff_L$. On morphisms it takes cartesian lifts. Let $f\colon L \to K$, then $\Cliff_f \colon \Cliff_K \to \Cliff_L$ must send a Clifford monoid $M$ to a Clifford monoid with idempotents $L$. By \cref{cor:E_fibration} this can be constructed as the pullback of $f$ and $\eta_M$ which yields, essentially, $\{(\ell,x) \in L \times M \mid xx^* = f(\ell)\}$ with componentwise multiplication.

   Applying \cref{cor:clifford_functor} we have that $\Cliff_L \simeq \HomC{\Cat}{L\op}{\Grp}$. All that remains is to show that this is a pseudonatural equivalence.
   \begin{cd}
    {\Cliff_K} \& {\HomC{\Cat}{K\op}{\Grp}} \\
	{\Cliff_L} \& {\HomC{\Cat}{L\op}{\Grp}}
	\arrow["{\simeq}", from=1-1, to=1-2]
	\arrow["{\Cliff_f}"', from=1-1, to=2-1]
	\arrow["{\HomC{\Cat}{f}{\Grp}}", from=1-2, to=2-2]
	\arrow["{\simeq}"', from=2-1, to=2-2]
\end{cd}
By chasing a monoid $M \in \Cliff_K$ around the diagram for each $f\colon L \to K$ we can easily obtain the desired isomorphisms, which are also quickly seen to be natural. Finally, checking that the necessary coherence laws hold is routine and unenlightening.
\end{proof}

\begin{remark}
    The objects of $\Groth{\HomC{\Cat}{(-)\op}{\Grp}}$ are pairs $(L,F)$ where $L$ is a meet-semilattice and $F\: L\op \to \Grp$. Its morphisms are pairs $(t,\alpha)\: (L,F) \to (K,G)$ where $t\: L \to K$ and $\alpha\: F \to G\circ t$ is a natural transformation. Note that this description is identical to the category described in \cite{pasku2011clifford}.
    
    The equivalence $S\: \Groth{\HomC{\Cat}{(-)\op}{\Grp}} \to \Cliff$ sends the object $(L,F)$ to Clifford monoid $\Groth{F}$ and the morphism $(t,\alpha) \: (L,F) \to (K,G)$ to the homomorphism $S(t,\alpha) \: \Groth{F} \to \Groth{G}$ with $S(t,\alpha)(e,x) = (t(e),\alpha_{e}(x))$.
\end{remark}

\begin{remark}
The results above are expected because in the last section we proved that Clifford monoids are the same as discrete fibrations in groups over semilattices. And, at least in the ordinary non-monoidal case, it is known that collecting together discrete fibrations forms a Grothendieck bifibration $\cod\:\Fib[d]\to \Cat_{\opn{s}}$, where $\Cat_{\opn{s}}$ is the 1-category of small categories and functors between them. Indeed, this is the Grothendieck construction of
\begin{fun}
	\Fib[d]_{(-)} & \: & \Cat_{\opn{s}}\op & \too & \Cat \\[1ex]
    && \X & \mto & \Fib[d]_{\X}\simeq \HomC{\Cat}{\X\op}{\Set} \\[1ex]
    && \fib[o]{\X}{F}{\Y} & \mto & \fib{\HomC{\Cat}{\X\op}{\Set}}{(-) \c F\op}{\HomC{\Cat}{\Y\op}{\Set}}
\end{fun}
Notice that the equivalences $\Fib[d]_{\X}\simeq \HomC{\Cat}{\X\op}{\Set}$ are pseudonatural in $\X$, as proved in \cite{CavMes24}. 
The fact that all the reindexing functors have a left adjoint
\begin{cd}[3.5][7]
    \HomC{\Cat}{\Y\op}{\Set}\ar[r,bend left=18,"{\Lan{F\op}}",""{name = 1}]\&  \HomC{\Cat}{\X\op}{\Set} \ar[l,bend left=18,"{(-)\c F\op}",""{name = 2}]
    \ar[from=1,to=2,adj]
\end{cd}
given by the left Kan extension ensures that $\cod\:\Fib[d]\to \Cat_{\opn{s}}$ is not only a fibration but also an opfibration, and thus a bifibration.
\end{remark}

We can also use the opfibrational structure of $E\: \Cliff \to \SLat$ to obtain a representation of the category of Clifford monoids. In fact, it is this representation that behaves better for the applications to inverse semirings that we present in Section \ref{secinvsemirings}.
From the theory of bifibrations, the pseudofunctor corresponding to this opfibration is given by taking the adjoint of the left action of $\HomC{\Cat}{-}{\Grp}$ on morphisms, which involves taking left Kan extensions. We might call this dual pseudofunctor ${\Cliff}_{(-)}^{\,\,*}\colon \SLat \to \Cat$.

We will be particularly interested in this construction in the case of \emph{commutative} Clifford monoids. In this setting, we can also recover the monoidal structure from another application of the monoidal Grothendieck construction.

\begin{theorem}\label{theorEismonoidalopfib}
    The opfibration $E\:\CInvMon\to \SLat$ is equivalent over $\SLat$ to a monoidal opfibration, i.e.\ an object of $\PsMon(\Fib[o])_{\SLat}$.
\end{theorem}
\begin{proof}
We have already seen that $E\colon \Cliff \to \SLat$ is an opfibration in \cref{cor:E_fibration}. The same proof shows this for the restriction to $\CInvMon$.

The functor can be seen to be strong monoidal by a simple calculation. Note that the kernel equivalence relation of $\eta_A\colon A \twoheadrightarrow E(A)$ is generated by $x \sim x + x$. Recall that the tensor product $A \otimes B$ has a presentation $\langle a \otimes b \mid 0 \otimes b = 0, (a + a') \otimes b = a \otimes b + a' \otimes b, a \otimes 0 = 0, a \otimes (b + b') = a \otimes b + a \otimes b'\rangle$ and that this can be computed on a presentation of $A$ and $B$. Doing this for $E(A) \otimes E(B)$ and using that idempotents are closed under sums gives the same presentation as $E(A \otimes B)$. The isomorphism sends $e \otimes e'$ in $E(A) \otimes E(B)$ to itself in $E(A \otimes B)$ in the obvious way.

The known theory of monoidal opfibrations actually requires the functors to be \emph{strict} monoidal. A more convenient version of the theory is possible by taking a weaker notion of morphism of fibrations. Nonetheless, we can choose the tensor product on $\CInvMon$ so that $E$ is actually strict monoidal by defining $a \otimes b$ (and products of these) for idempotents to agree with tensors of the semilattices and building the rest of the tensor product around these.

It remains to show that opcartesian morphisms are stable under tensor product. By functoriality of $\otimes$ it suffices to show this for when one of the morphisms is the identity and by symmetry we may assume it to be the second one. Note that opcartesian lifts are computed by pushouts as in the following diagram.
\[\begin{tikzcd}
    A & {A +_{E(A)} L} \\
    {E(A)} & L
    \arrow[from=1-1, to=1-2]
    \arrow["\lrcorner"{anchor=center, pos=0.125, rotate=-90}, draw=none, from=1-2, to=2-1]
    \arrow["{\epsilon_A}", from=2-1, to=1-1]
    \arrow["f"', from=2-1, to=2-2]
    \arrow[from=2-2, to=1-2]
\end{tikzcd}\]
Tensoring with $B$ and using that this preserves colimits gives $A \otimes B +_{E(A) \otimes B} L \otimes B$. We want this to agree with $A \otimes B +_{E(A \otimes B)} L \otimes B \cong A \otimes B +_{E(A) \otimes E(B)} L \otimes B$. Both of these are given by a quotient of the coproduct $(A \otimes B) \times (L \otimes B)$. In the former case, the congruence is generated by $(e \otimes b, 0) \sim (0, f(e) \otimes b)$ for $e \in E(A)$ and $b \in B$. In the latter case, it is generated by $(e \otimes e', 0) \sim (0, f(e) \otimes e')$ for $e \in E(A)$ and $e' \in E(B)$. But $e \otimes b = (e-e) \otimes b = e \otimes b - e \otimes b = e \otimes (b-b) = e \otimes e_b$ and so these agree.
\end{proof}

\begin{theorem}\label{thmtotalccliff}
    The pseudofunctor
    \begin{fun}
	\CInvMon_{(-)}^* & \: & \SLat & \too & \Cat \\[1ex]
    && L & \mto & \CInvMon_L\simeq \HomC{\Cat}{L\op}{\Ab} \\[1ex]
    && \fib{L}{F}{L'} & \mto & \fib{\HomC{\Cat}{L\op}{\Ab}}{\Lan{F\op}}{\HomC{\Cat}{L'\op}{\Ab}}
\end{fun}
    extends to a weakly lax monoidal pseudofunctor, and 
    $E\:\CInvMon\to \SLat$ is equivalent over $\SLat$ to its monoidal Grothendieck construction.
\end{theorem}
\begin{proof}
    Note that \cref{thm:main_equivalence} easily restricts to the commutative case and so we have that $E\:\CInvMon\to \SLat$ is equivalent over $\SLat$ to the Grothendieck construction of $\CInvMon_{(-)}^*$.
    Then, by the monoidal Grothendieck construction, there is an equivalence of categories
    $$\PsMon(\Fib[o])_{\SLat}\simeq \HomC{\MonTwoCatps}{\SLat}{\Cat}$$
    So the structure of monoidal opfibration constructed in \thex\ref{theorEismonoidalopfib} is transported into a structure of weakly lax monoidal pseudofunctor.
\end{proof}

\begin{cons}\label{consweaklylaxmonpseudofunstruct}
    We can extract from \cref{thmtotalccliff} the explicit structure of weakly lax monoidal pseudofunctor of $\CInvMon_{(-)}^*$. The structure pseudonatural transformation regulating the image of the tensor product
    \begin{cd}[6][19] 
        {\SLat\times \SLat} \& {\Cat\times \Cat} \\
	\SLat \& \Cat
	\arrow["{\CInvMon_{(-)}^* \times \CInvMon_{(-)}^*}", from=1-1, to=1-2]
	\arrow["\otimes"', from=1-1, to=2-1]
	\arrow["\phi", Rightarrow, from=1-2, to=2-1, shorten <=4ex, shorten >=4ex]
	\arrow["\times", from=1-2, to=2-2]
	\arrow["{\CInvMon_{(-)}^*}"', from=2-1, to=2-2]
    \end{cd}
    has component on $(L,L')$ given by the functor
    $$\phi_{L,L'}\:\HomC{\Cat}{L\op}{\Ab}\times \HomC{\Cat}{{L'}\op}{\Ab}\to \HomC{\Cat}{(L\otimes L')\op}{\Ab}$$
    that sends $(F\:L\op\to \Ab,\, G\:{L'}\op\to \Ab)$ to the functor $(L\otimes L')\op\to \Ab$ that corresponds to the global tensor product (with respect to the structure of monoidal opfibration) of the commutative inverse monoids associated to $F$ and $G$. It is straightforward to see
    that the obtained functor $(L\otimes L')\op\to \Ab$ is precisely the functor induced by the tensor product of $\SLat$ as in the following diagram
    \begin{cd}[5.5][6]
        {(L\times L')}\op \& {\Ab\times \Ab} \\
	{(L\otimes L')}\op \& \Ab
	\arrow["{F\x G}", from=1-1, to=1-2]
	\arrow[from=1-1, to=2-1]
	\arrow["{\otimes_{\Ab}}", from=1-2, to=2-2]
	\arrow["{\phi_{L,L'}(F,G)}"', dashed, from=2-1, to=2-2]
    \end{cd}

    The structure pseudonatural transformation regulating the image of the unit
    \begin{cd}[5.5][6]
        \1 \& \\
	\SLat \& \Cat
	\arrow["I"', from=1-1, to=2-1]
	\arrow[""{name=0, anchor=center, inner sep=0}, "1", from=1-1, to=2-2]
	\arrow["{\CInvMon_{(-)}^*}"', from=2-1, to=2-2]
	\arrow["\iota", Rightarrow, from=0, to=2-1,shorten <=0.5ex,shorten >=0.5ex]
    \end{cd}
    is given by the functor
    $$\iota\:1\to \HomC{\Cat}{{I}\op}{\Ab}$$
    that picks the functor $\mathbbm{2}\op\to \Ab$ associated with the commutative inverse monoid $\Z_0$, which is the global unit in $\CInvMon$.
\end{cons}

\section{Factorisation systems for Clifford monoids}

The fibration $\Cliff \to \SLat$ can be leveraged to immediately deduce structural characteristics of the category $\Cliff$. One example of this is how we can use it to define various \emph{factorisation systems}.

\begin{definition}
    A factorisation system on a category $\C$ is a pair $(\E,\M)$ of classes of morphisms such that
    \begin{enumerate}
        \item every morphism $f$ in $\C$ factorises as $me$ where $m \in \M$ and $e \in \E$ uniquely up to unique isomorphism,
        \item the classes $\E$ and $\M$ are each closed under composition and contain all isomorphisms.
    \end{enumerate}
\end{definition}

Classic factorisation systems in algebra include those where the left class is the surjective homomorphisms (regular epimorphisms) and the right class is the injective homomorphisms (monomorphisms). Every category also admits two trivial factorisation systems $(\mathrm{Iso},\mathrm{All})$ and $(\mathrm{All},\mathrm{Iso})$ where one class consists only of isomorphisms and the other contains all morphisms.

(Op)fibrations allow us to lift factorisation systems from the base category to the total category (see \cite[\S 21]{joyofcats}). Given a factorisation system $(\E,\M)$ on $\SLat$, using that $E\colon \Cliff \to \SLat$ is an opfibration, we can induce a factorisation system $(\E',\M')$ on $\Cliff$ where $\E'$ consists of the opcartesian morphisms lying over morphisms in $\E$ and $\M'$ consists of the morphisms lying over morphisms in $\M$. Moreover, using that $E$ is a fibration, we obtain a factorisation system $(\E'', \M'')$ where $\E''$ is the class of morphisms lying over $\E$ and $\M''$ is the class of cartesian morphisms lying over $\M$.

Starting with the trivial class $(\mathrm{All}, \mathrm{Iso})$ the first construction gives a factorisation system where the right class is the homomorphisms that induce isomorphism on the semilattices of idempotents and the left class is the morphisms $f\colon M \to N$ such that the naturality square
\[\begin{tikzcd}
	M & N \\
	{E(M)} & {E(N)}
	\arrow["f", from=1-1, to=1-2]
	\arrow["{\epsilon_M}", hook, from=2-1, to=1-1]
	\arrow["{E(f)}"', from=2-1, to=2-2]
	\arrow["{\epsilon_N}"', hook, from=2-2, to=1-2]
\end{tikzcd}\]
is a pushout.

Starting with $(\mathrm{Iso}, \mathrm{All})$, the second construction gives a factorisation system where now the left class is given by the homomorphisms which induce isomorphisms on the semilattices of idempotents and, by \cref{lem:cartesian_characterisation}, the right class consists of morphisms $f\colon M \to N$ such that $f$ restricts to group isomorphisms $G_e \to G_{f(e)}$ for each $e \in E(M)$.

If we instead start with the $(\mathrm{RegEpi}, \mathrm{Mono})$-factorisation system on $\SLat$, the first construction now gives a factorisation system where the right class consists of the morphisms which are injective on idempotents. This is an important class of morphisms in semigroup theory --- the \emph{idempotent-separating morphisms}. The left class is those morphisms $f\colon M \to N$ such that $E(f)$ is surjective and the square
\[\begin{tikzcd}
	M & N \\
	{E(M)} & {E(N)}
	\arrow["f", from=1-1, to=1-2]
	\arrow["{\epsilon_M}", hook, from=2-1, to=1-1]
	\arrow["{E(f)}"', from=2-1, to=2-2]
	\arrow["{\epsilon_N}"', hook, from=2-2, to=1-2]
\end{tikzcd}\]
is a pushout. This means that $f$ itself is also a surjection and the quotient of the idempotents induces the quotient $f$. In other words, the left class consists of the quotients whose kernel equivalence relations are generated by pairs of idempotents. (In fact, there is a similar factorisation system on the entire category of inverse monoids, but we omit the details.)

Finally, applying the second construction we obtain a factorisation system where the left class is those morphisms which are surjective on idempotents and the right class is the injections which induce group isomorphisms $G_e \to G_{f(e)}$ for each idempotent $e$.

\section{Monoids and the monoidal Grothendieck construction}\label{sec:monoids_in_monoidal_fibrations}

In this section, we present a new result about the monoidal Grothendieck construction. We prove that the category of monoids in the total category of a monoidal opfibration is given by the Grothendieck construction of a pseudofunctor that takes monoids in the fibres. The proof gives an explicit construction for endowing the relevant fibres with a natural monoidal structure. As this is a general result about monoidal opfibrations, we expect it will have utility beyond just the theory of Clifford monoids.

First, recall the following known proposition (for example, see \cite[Proposition 2.6]{MoeVas20}). Here we are taking pseudofunctors and pseudonatural transformations, rather than the strict versions, but nothing changes since the domain is $\1$.

\begin{proposition}
    The 2-category of internal pseudo-monoids in a monoidal 2-category $\K$ is biequivalent to the 2-category $\HomC{\MonTwoCatps}{\1}{\K}$ of weakly lax monoidal pseudofunctors from $\1$ to $\K$, weakly monoidal pseudonatural transformations and monoidal modifications.
\end{proposition}
\begin{proof}[Proof (sketch)]
    Consider a weakly lax monoidal pseudofunctor $\ov{M}\:\1\to \K$. The image of $\ov{M}$ on the unique element $\ast\in \1$ gives an object $M$ in $\K$. The weakly lax monoidal structure of $\ov{M}$ then equips $M$ with a structure given by the pseudonatural transformations
    \begin{eqD*}
    \begin{cd}*
        {\1\x \1} \&\& \1 \\
	{\K\x \K} \&\& \K
	\arrow["\otimes", from=1-1, to=1-3]
	\arrow[""{name=0, anchor=center, inner sep=0}, "{\ov{M}\x \ov{M}}"', from=1-1, to=2-1]
	\arrow[""{name=1, anchor=center, inner sep=0}, "{\ov{M}}", from=1-3, to=2-3]
	\arrow["\otimes"', from=2-1, to=2-3]
	\arrow["m"', shorten <=4.7ex, shorten >=4.2ex, Rightarrow, from=0, to=1]
    \end{cd}\quad \h[6] \quad
    \begin{cd}*
    \1 \& \1 \\
	\& \K
	\arrow["{\id{}}", from=1-1, to=1-2]
	\arrow["e", shift right=7, Rightarrow, from=1-1, to=1-2, shorten <=1ex]
	\arrow["I"', from=1-1, to=2-2,bend right]
	\arrow["{\ov{M}}", from=1-2, to=2-2]
    \end{cd}
    \end{eqD*}
    together with invertible modifications that represent an associator and two unitors. Notice that $m$ and $e$ correspond to morphisms $m\:M\otimes M\to M$ and $e\:I\to M$ in $\K$. It is straightforward to check that this data is precisely that of a pseudo-monoid in $\K$.
\end{proof}

\begin{proposition}\label{propinducepsmon}
    Let $\A$ be a monoidal category. Every weakly lax monoidal pseudofunctor $F\:\A\to \Cat$ induces a pseudofunctor $\widehat{F}\:\Monf(\A)\to \MonCat$. More precisely, $\widehat{F}$ is given by the composite
    $$\Monf(\A)\simeq\HomC{\MonTwoCatps}{\1}{\A}\aar{F\c -}\HomC{\MonTwoCatps}{\1}{\Cat}\simeq \MonCat$$
\end{proposition}
\begin{proof}
    This is due to the fact that $\MonTwoCatps$ is a tricategory and thus 
    $$\HomC{\MonTwoCatps}{\1}{-}\:\MonTwoCatps\to \Bicat$$
    is a trifunctor.
\end{proof}

\Cref{propinducepsmon} provides us with an interesting link between the two faces of the monoidal Grothendieck construction. In \cite{MoeVas20}, it is shown that when the base monoidal category $\A$ is cocartesian monoidal, then weakly lax monoidal pseudofunctors from $\A$ to $\Cat$ are the same as pseudofunctors from $\A$ to $\MonCat$, and the two faces of the monoidal Grothendieck construction coincide. \prox\ref{propinducepsmon} shows that even when the base is not cocartesian monoidal we can produce a pseudofunctor into $\MonCat$ starting from a weakly lax monoidal pseudofunctor into $\Cat$.

When the base $\A$ is cocartesian monoidal, the category of monoids in $\A$ is isomorphic to $\A$. \prox\ref{propinducepsmon} thus induces a pseudofunctor $\A\to \MonCat$. This could thus lead to a new proof of the result of \cite{MoeVas20}.

\begin{cons}\label{consinducemonstr}
    We can extract from \prox\ref{propinducepsmon} the following explicit construction of $\widehat{F}\:\Monf(\A)\to \MonCat$. Consider a monoid $(M,M\otimes M \ar{m}M,I\ar{e}M)$ in $\A$, exhibited by a weakly lax monoidal pseudofunctor $\ov{M}\:\1\to \A$. Postcomposing with $F\:\A\to \Cat$ we obtain a monoidal category, which we now describe. In the following, we omit the isomorphisms given by the pseudofunctoriality of $F$.
    The base category is given by $F(\ov{M}(\ast)) = F(M)$. Recall that $F(M)$ will become the fibre over the monoid $M$ in $\Groth{\widehat{F}} \to \Monf(\A)$. We now equip it with a monoidal structure. The tensor product $\otimes_{F(M)}$ and monoidal unit $I_{F(M)}$ of $F(M)$ are given by the pastings
    \begin{eqD*}
    \begin{cd}*
    {\1\x \1} \&\& \1 \\
	{\A\x \A} \&\& \A \\
	{\Cat\x \Cat} \&\& \Cat
	\arrow["\otimes", from=1-1, to=1-3]
	\arrow[""{name=0, anchor=center, inner sep=0}, "{\ov{M}\x \ov{M}}"', from=1-1, to=2-1]
	\arrow[""{name=1, anchor=center, inner sep=0}, "{\ov{M}}", from=1-3, to=2-3]
	\arrow["\otimes"', from=2-1, to=2-3]
	\arrow[""{name=2, anchor=center, inner sep=0}, "{F\x F}"', from=2-1, to=3-1]
	\arrow[""{name=3, anchor=center, inner sep=0}, "F", from=2-3, to=3-3]
	\arrow["\x"', from=3-1, to=3-3]
	\arrow["m"', shorten <=4.7ex, shorten >=4.2ex, Rightarrow, from=0, to=1]
	\arrow["{\phi^F}"', shorten <=4.7ex, shorten >=4.2ex, Rightarrow, from=2, to=3]
    \end{cd}\quad \h[18]\quad
    \begin{cd}*[6][8]
    \1 \& \1 \\
	\& \A \\
	\& \Cat
	\arrow["{\id{}}", from=1-1, to=1-2]
	\arrow["e"{pos=0.59}, shift right=7, Rightarrow, from=1-1, to=1-2, shorten <=2.9ex]
	\arrow["{\iota^F}"{pos=0.69}, shift right=22, Rightarrow, from=1-1, to=1-2, shorten <=4.5ex]
	\arrow["I"', curve={height=12pt}, from=1-1, to=2-2]
	\arrow["\1"', curve={height=18pt}, from=1-1, to=3-2]
	\arrow["{\ov{M}}", from=1-2, to=2-2]
	\arrow["F", from=2-2, to=3-2]
    \end{cd}
    \end{eqD*}
    where $\phi^F$ and $\iota^F$ are the pseudonatural transformations given by the weakly lax monoidal structure of $F$. Explicitly, $\otimes_{F(M)}$ and $I_{F(M)}$ are represented respectively by the morphisms
    \begin{center}
    \vspace{-10pt}
    \[F(M)\x F(M)\aar{\phi^F_{M,M}} F(M\otimes M)\aar{F(m)}F(M)\]
    and
    \[\1\aarr{\iota^F}F(I)\aar{F(e)}F(M).\]
    \end{center}
    It is straightforward to calculate that the associator and unitors of $F(M)$ are given, respectively, by the pastings of natural transformations
\begin{cd}[3.9][-8.4]
	\&\&\&[-3ex] {F(M) \times F(M) \times F(M)} \&[-3ex]\&\& \\[-1ex]
	\& {F(M) \times F(M \otimes M)} \&\& {\cong \omega^F_{M,M,M}} \&\& {F(M \otimes M) \times F(M)} \\
	{F(M) \times F(M)} \& {\overset{\text{interch}}{\cong}} \& {F(M \otimes (M \otimes M))} \&\& {F((M \otimes M) \otimes M)} \& {\overset{\text{interch}}{\cong}} \& {F(M) \times F(M)} \\
	\& {F(M \otimes M)} \&\&\&\& {F(M \otimes M)} \\[-1ex]
	\&\&\& {F(M)}
	\arrow["{\id{} \times \varphi^F_{M,M}}"', from=1-4, to=2-2]
	\arrow["{\varphi^F_{M,M} \times \id}", from=1-4, to=2-6]
	\arrow["{\id \times F(m)}"', from=2-2, to=3-1]
	\arrow["{\varphi^F_{M, M \otimes M}}"{description}, from=2-2, to=3-3]
	\arrow["{\varphi^F_{M \otimes M, M}}"{description}, from=2-6, to=3-5]
	\arrow["{F(m) \times \id}", from=2-6, to=3-7]
	\arrow["{\varphi^F_{M,M}}"', from=3-1, to=4-2]
	\arrow["{F(\alpha^{\A})}","{\iso}"', from=3-5, to=3-3]
	\arrow["{F(\id \otimes m)}"{description}, from=3-3, to=4-2]
	\arrow["{F(m \otimes \id)}"{description}, from=3-5, to=4-6]
	\arrow["{\varphi^F_{M,M}}", from=3-7, to=4-6]
	\arrow["{F(m)}"', from=4-2, to=5-4]
	\arrow["{F(m)}", from=4-6, to=5-4]
\end{cd}

\begin{cd}
    {F(M)} \& {F(M) \times \1} \& {F(M) \times F(I)} \& {F(M) \times F(M)} \\
	\&\& {F(M \otimes I)} \& {F(M \otimes M)} \\[-3ex]
	\& {F(M)}
	\arrow["\cong", from=1-1, to=1-2]
	\arrow["{F(\rho^{\A})}"{pos=0.4}, from=2-3, to=1-1]
	\arrow["\id"', curve={height=6pt}, from=1-1, to=3-2]
	\arrow["{\id \times \iota^F}", from=1-2, to=1-3]
	\arrow["{\xi^F}"'{pos=0.7}, shift right=5, iso, from=1-2, to=1-3]
	\arrow["{\id \times F(e)}", from=1-3, to=1-4]
	\arrow["{\overset{\text{interch}}{\cong}}"{marking, allow upside down}, shift right=9, draw=none, from=1-3, to=1-4]
	\arrow["{\varphi^F_{M,I}}", from=1-3, to=2-3]
	\arrow["{\varphi^F_{M,M}}", from=1-4, to=2-4]
	\arrow["{F(\id \otimes e)}"', from=2-3, to=2-4]
	\arrow["{F(m)}", curve={height=-12pt}, from=2-4, to=3-2]
\end{cd}
    and a pasting analogous to the latter for the other unitor. Above, $\omega^F$ and $\xi^F$ denote the invertible modifications given by the weakly lax monoidal structure of $F$. The isomorphic 2-cells labelled ``interch'' are given by the weak interchange rule of $\Bicat$, whereas the diagrams without filling commute, up to the omitted isomorphisms given by the pseudofunctoriality of $F$.
\end{cons}

\begin{remark}\label{remglobalmonstr}
    The pseudonatural transformations $\phi^F$ and $\iota^F$ of \consx\ref{consinducemonstr} are precisely what equips the total monoidal category of the monoidal Grothendieck construction of a weakly lax monoidal pseudofunctor into $\Cat$ with a global monoidal structure. Indeed, this is explained in \cite{MoeVas20}. So, interestingly, we can view the induced monoidal structure on the fibres that we have constructed in \consx\ref{consinducemonstr} as a structure that considers the global monoidal structure in the total category and then takes opcartesian liftings to ensure that we land over $M$.
\end{remark}

We now present the main theorem of this section. We prove that monoids in the total monoidal category given by a monoidal Grothendieck construction precisely correspond to the Grothendieck construction of a pseudofunctor that takes monoids in the fibres.

\begin{theorem}\label{theormoninmonoidalgroth}
    Let $P\:\E\to \A$ be a monoidal opfibration, given by the monoidal Grothendieck construction of a weakly lax monoidal pseudofunctor $F\:\A\to \Cat$. Then $\Monf(P)\:\Monf(\E)\to \Monf(\A)$ is the Grothendieck construction of the pseudofunctor
    $$\Monf(\A)\arr{\widehat{F}}\MonCat\aar{\Monf(-)}\Cat$$
\end{theorem}
\begin{proof}
    We prove that $\Monf(\E)$ is isomorphic over $\Monf(\A)$ to the Grothendieck construction of $\Monf(-)\c \widehat{F}$. An object of $\Monf(\E)=\Monf(\Groth{F})$ is given by an object $(M,X)\in \Groth{F}$ --- that is, a pair of $M\in \A$ and $X\in F(M)$ --- equipped with a multiplication map $(m,\mu)\:(M,X)\xgf (M,X)\to (M,X)$ and a unit map $(e,\epsilon)\:I_{\Groth{F}}\to (M,X)$ in $\Groth{F}$ satisfying the axioms of an internal monoid. As explained in \cite{MoeVas20} (see also \remx\ref{remglobalmonstr} as well as \defx\ref{defmonoidalfibration} and the discussion below it), the global tensor product $\xgf$ of the total category of the monoidal Grothendieck construction of a weakly lax pseudofunctor into $\Cat$ is calculated as follows:
    $$(M,X)\xgf (N,Y)=(M\otimes N,X\xf Y)$$
    where $\otimes$ is the tensor product in $\A$ and $\xf$ coincides with the component $\phi^F_{M,N}$ on $(M,N)$ of the pseudonatural transformation $\phi^F$ given by the weakly lax monoidal structure of $F$.
    
    Explicitly, $(m,\mu)\:(M\otimes M,X\xf X)\to (M,X)$ is a morphism
    in $\Groth{F}$ and so is given by a morphism $m\:M\otimes M\to M$ in $\A$ and a morphism 
    $\mu\:F(m)(X\xf X)\to X$
    in $F(M)$. Similarly, the monoidal unit of the global monoidal structure of $\Groth{F}$ is
    $I_{\Groth{F}}=(I,I_F)$
    where $I$ is the monoidal unit of $\A$ and $I_F=\iota^F(\ast)$ where $\iota^F$ is the pseudonatural transformation given by the weakly lax monoidal structure of $F$. Finally, $(e,\epsilon)\:(I,I_F)\to (M,X)$ is a morphism 
    in $\Groth{F}$ given by a morphism $e\:I\to M$ in $\A$ and a morphism
    $\epsilon\:F(e)(I_F)\to X$
    in $F(M)$. 
    
    The axioms of internal monoids in $\Groth{F}$ are given by the following commutative diagrams.
    \begin{cd}[5][8]
    {((M,X) \otimes_{\Groth{F}} (M,X)) \otimes_{\Groth{F}} (M,X)} \& {(M,X) \otimes_{\Groth{F}} (M,X)} \\[-2ex]
	{(M,X) \otimes_{\Groth{F}} ((M,X) \otimes_{\Groth{F}} (M,X))} \\
	{(M,X) \otimes_{\Groth{F}} (M,X)} \& {(M,X)}
	\arrow["{(m,\mu) \otimes_{\Groth{F}} \id}", from=1-1, to=1-2]
	\arrow["\cong"', from=1-1, to=2-1]
	\arrow["{(m,\mu)}", from=1-2, to=3-2]
	\arrow["{\id \otimes_{\Groth{F}} (m,\mu)}"', from=2-1, to=3-1]
	\arrow["{(m,\mu)}"', from=3-1, to=3-2]
    \end{cd}
    \begin{eqD*}
    \begin{cd}*
        {(M,X)} \& {(I, I_F) \otimes_{\Groth{F}} (M,X)} \\
	{(M,X)} \& {(M,X) \otimes_{\Groth{F}} (M,X)}
	\arrow["\cong", from=1-1, to=1-2]
	\arrow[equals, from=1-1, to=2-1]
	\arrow["{(e,\varepsilon) \otimes_{\Groth{F}} \id}", from=1-2, to=2-2]
	\arrow["{(m,\mu)}", from=2-2, to=2-1]
    \end{cd}\quad \h[16]\quad
        \begin{cd}*
{(M,X)} \& {(M,X) \otimes_{\Groth{F}} (I,I_F)} \\
	{(M,X)} \& {(M,X) \otimes_{\Groth{F}} (M,X)}
	\arrow["\cong", from=1-1, to=1-2]
	\arrow[equals, from=1-1, to=2-1]
	\arrow["{\id \otimes_{\Groth{F}} (e,\epsilon)}", from=1-2, to=2-2]
	\arrow["{(m,\mu)}", from=2-2, to=2-1]   
        \end{cd}
    \end{eqD*}
    It is straightforward to prove that on the first component these mean that $(M,m,e)$ is a monoid in $\A$, while on the second component they encode that $(X,\mu,\epsilon)$ is a monoid in $F(M)$ with respect to the induced monoidal structure we described in \consx\ref{consinducemonstr}. For this we use that $\otimes_{F(M)}$ and $I_{F(M)}$ are respectively given by
    \begin{center}
    $$F(M)\x F(M)\aar{\xf} F(M\otimes M)\aar{F(m)}F(M)$$
    and
    $$\1\aarr{I_F}F(I)\aar{F(e)}F(M)$$
    \end{center}
    due to \remx\ref{remglobalmonstr}. Then $\mu$ and $\epsilon$ are morphisms $\mu\:X\xfm X\to X$ and $\epsilon\:I_{F(M)}\to X$.
    
    It is also helpful to consider the fibrational point of view, for which 
    $X\otimes_{F(M)}X$ is given by the opcartesian lifting of $m\:M\otimes M\to M$ to $X\otimes_F X$
    \begin{cd}
        \E \& X \&[-2ex] {X\otimes_F X} \& {X\otimes_{F(M)} X} \\
	\A \& M \& {M\otimes M} \& M
	\arrow["P"', from=1-1, to=2-1]
	\arrow[maps to, from=1-2, to=2-2]
	\arrow[dashed, from=1-3, to=1-4]
	\arrow[maps to, from=1-3, to=2-3]
	\arrow[maps to, from=1-4, to=2-4]
	\arrow["m"', from=2-3, to=2-4]
    \end{cd}
    and similarly $I_{F(M)}$ is given by the opcartesian lifting of $e$ to $I_F$
    \begin{cd}
    \E \& {I_F} \& {I_{F(M)}} \\
	\A \& I \& M
	\arrow["P"', from=1-1, to=2-1]
	\arrow[dashed, from=1-2, to=1-3]
	\arrow[maps to, from=1-2, to=2-2]
	\arrow[maps to, from=1-3, to=2-3]
	\arrow["e"', from=2-2, to=2-3]
    \end{cd}
    By the universal property of the lift, the multiplication map $\mu\:X\otimes_{F(M)} X\to X$ and the unit map $\epsilon\:I_{F(M)}\to X$, which are vertical over $M$, correspond respectively to morphisms $[\mu]$ and $[\epsilon]$ as in the following diagram.
    \begin{cd}
        \E \& {X\otimes_F X} \& X \& {I_F} \& X \\
	\A \& {M\otimes M} \& M \& I \& M
	\arrow["P"', from=1-1, to=2-1]
	\arrow["{[\mu]}", from=1-2, to=1-3]
	\arrow[maps to, from=1-2, to=2-2]
	\arrow[maps to, from=1-3, to=2-3]
	\arrow["{[\epsilon]}", from=1-4, to=1-5]
	\arrow[maps to, from=1-4, to=2-4]
	\arrow[maps to, from=1-5, to=2-5]
	\arrow["m"', from=2-2, to=2-3]
	\arrow["e"', from=2-4, to=2-5]
    \end{cd}
    We can confirm the monoid axioms using the universal properties.
    
    We have thus shown that objects in $\Monf(\E)$ precisely correspond to objects of $\Groth{(\Monf(-)\c \widehat{F})}$. Indeed the latter objects are pairs $((M,m,e),(X,\mu,\epsilon))$ of a monoid $(M,m,e)$ in $\A$ and a monoid $(X,\mu,\epsilon)$ in $F(M)$ with respect to the induced monoidal structure.
    We now look at the correspondence on morphisms.
    
    Consider a morphism of monoids in $\E=\Groth{F}$ given by
    $$(f,h)\:((M,X),(m,\mu),(e,\epsilon))\to ((N,Y),(m',\mu'),(e',\epsilon')).$$
    Here $f\:M\to N$ is a morphism in $\A$ and
    $h\:F(f)(X)\to Y$ is a morphism in $F(N)$ such that the following diagrams commute.
    \begin{eqD*}
    \begin{cd}*
        {(M\otimes M, X\xf X)} \& {(M,X)} \\
	{(N\otimes N, Y\xf Y)} \& {(N,Y)}
	\arrow["{(m,\mu)}", from=1-1, to=1-2]
	\arrow["{(f\otimes f, h\xf h\c j)}"', from=1-1, to=2-1]
	\arrow["{(f,h)}", from=1-2, to=2-2]
	\arrow["{(m',\mu')}"', from=2-1, to=2-2]
    \end{cd}\quad \h[16]\quad
    \begin{cd}*
        {(I,I_F)} \& {(M,X)} \\
	\& {(N,Y)}
	\arrow["{(e,\epsilon)}", from=1-1, to=1-2]
	\arrow["{(e',\epsilon')}"', curve={height=6pt}, from=1-1, to=2-2]
	\arrow["{(f,h)}", from=1-2, to=2-2]
    \end{cd}
    \end{eqD*}
    Here $j$ is the isomorphism
    $$F(f\otimes f)(X\xf X)\iso F(f)(X)\xf F(f)(X)$$
    given by the pseudonaturality of $\phi^F$. On the first component, the two diagrams above precisely entail that $f\:(M,m,e)\to (N,m',e')$ is a morphism of monoids in $\A$, whereas on the second component, they translate to give
    \begin{eqD*}
    \begin{cd}*[3.5][11.0]
        {F(f)(F(m)(X\xf X))} \& {F(f)(X)} \&[-8ex] \\
	{F(m')(F(f\otimes f)(X\xf X))} \&\& Y \\
	{F(m')(F(f)(X)\xf F(f)(X))} \& {F(m')(Y\xf Y)}
	\arrow["{F(f)(\mu)}", from=1-1, to=1-2]
	\arrow["\cong"', from=1-1, to=2-1]
	\arrow["h", from=1-2, to=2-3]
	\arrow["{F(m')(j)}"', from=2-1, to=3-1]
	\arrow["{F(m')(h\xf h)}"', from=3-1, to=3-2]
	\arrow["{\mu'}"', from=3-2, to=2-3]
    \end{cd}
    \end{eqD*}
    \begin{eqD*}
    \begin{cd}*[3.5][5.0]
        {F(f)(F(e)(I_F))} \& {F(f)(X)} \\
	{F(e')(I_F)} \& Y
	\arrow["{F(f)(\epsilon)}", from=1-1, to=1-2]
	\arrow["\cong"', from=1-1, to=2-1]
	\arrow["h", from=1-2, to=2-2]
	\arrow["{\epsilon'}"', from=2-1, to=2-2]
    \end{cd}
    \end{eqD*}
    where the unlabelled isomorphisms are given by the pseudofunctoriality of $F$. 

    Consider now a morphism in $\Groth{(\Monf(-)\c \widehat{F})}$
    $$(f,k)\:((M,m,e),(X,\mu,\epsilon))\to ((N,m',e'),(Y,\mu',\epsilon')).$$
    This is given by a morphism $f\:(M,m,e)\to (N,m',e')$ between monoids in $\A$ and a morphism
    $k\:(\Monf(-)\c \widehat{F})(f)(X,\mu,\epsilon)\to (Y,\mu',\epsilon')$
    between monoids in $F(N)$ (with respect to the induced monoidal structure described in \consx\ref{consinducemonstr}). The strong monoidal functor $\widehat{F}(f)$ is given by the whiskering
    \begin{cd}[5][8]
    \1 \& \A \& \Cat
	\arrow[""{name=0, anchor=center, inner sep=0}, "{\ov{M}}", curve={height=-12pt}, from=1-1, to=1-2]
	\arrow[""{name=1, anchor=center, inner sep=0}, "{\ov{N}}"', curve={height=12pt}, from=1-1, to=1-2]
	\arrow["F", from=1-2, to=1-3]
	\arrow["{\ov{f}}", shorten <=0.7ex, shorten >=0.5ex, Rightarrow, from=0, to=1]
    \end{cd}
    where $\ov{f}$ is the weakly monoidal pseudonatural transformation corresponding to the morphism $f$ of monoids in $\A$. In other words, $\widehat{F}(f)$ is the functor $F(f)\:F(M)\to F(N)$ equipped with the structure of a strong monoidal functor given by the pseudofunctoriality of $F$ and the pseudonaturality of $\phi^F$. The object $(\Monf(-)\c \widehat{F})(f)(X,\mu,\epsilon)$ in $\Monf(F(N))$ is then given by $F(f)(X)$ equipped with multiplication $[\mu]$ given by
\[
\begin{tikzcd}
F(f)(X)\otimes_{F(N)} F(f)(X) = F(m')\bigl(F(f)(X)\xf F(f)(X)\bigr)
  \arrow[d, "F(m')(j^{-1})", "\scriptstyle\simeq"'] \\
F(m')\bigl(F(f\otimes f)(X\xf X)\bigr)
  \arrow[d, "\scriptstyle\simeq"'] \\
F(f)\bigl(F(m)(X\xf X)\bigr)
  \arrow[d, "F(f)(\mu)"] \\
F(f)(X)
\end{tikzcd}
\]
and unit
$$[\epsilon]\: I_{F(N)}=F(e')(I_F)\iso F(f)(F(e)(I_F))\aar{F(f)(\epsilon)} F(f)(X),$$ where the unlabelled isomorphisms are given by the pseudofunctoriality of $F$. So $k$ is a morphism from $F(f)(X)$ to $Y$ in $F(N)$ that satisfies the axioms of a monoid morphism in $F(N)$ given below.
\begin{eqD*}
\begin{cd}*
    {F(f)(X)\otimes_{F(N)} F(f)(X)} \& {F(f)(X)} \\
	{Y\otimes_{F(N)} Y} \& Y
	\arrow["{[\mu]}", from=1-1, to=1-2]
	\arrow["{k \otimes_{F(N)} k}"', from=1-1, to=2-1]
	\arrow["k", from=1-2, to=2-2]
	\arrow["{\mu'}"', from=2-1, to=2-2]
\end{cd}\quad\h[16]\quad
\begin{cd}*
    {I_{F(N)}} \& {F(f)(X)} \\
	\& Y
	\arrow["{[\epsilon]}", from=1-1, to=1-2]
	\arrow["{\epsilon'}"', curve={height=6pt}, from=1-1, to=2-2]
	\arrow["k", from=1-2, to=2-2]
\end{cd}
\end{eqD*}

We can now conclude that monoid morphisms in $\Groth{F}$ precisely correspond to morphisms in $\Groth{(\Monf(-)\c \widehat{F})}$.

It is straightforward to show that the constructions we have produced above are functorial and mutually inverse, yielding an isomorphism of categories $\Monf(\E)\iso \Groth{(\Monf(-)\c \widehat{F})}$. Moreover, such an isomorphism of categories clearly commutes with the maps to $\Monf(\A)$ by construction.
\end{proof}

\begin{remark}
Of course, \thex\ref{theormoninmonoidalgroth} also yields a dual result about taking comonoids starting from a monoidal fibration which is given by the monoidal Grothendieck construction of a weakly lax monoidal pseudofunctor $\A\op\to \Cat$.
\end{remark}

\section{Application to inverse semirings}\label{secinvsemirings}

In this section, we apply \thex\ref{theormoninmonoidalgroth} to the monoidal opfibration of commutative inverse monoids in order to capture inverse semirings. We prove that the category of inverse semirings arises from the Grothendieck construction of a pseudofunctor that takes monoids in the fibres of the fibration of commutative Clifford monoids over semilattices. Interestingly, the induced monoidal structure on such fibres is precisely that of the Day convolution.

Recall that if $\C$ is a small monoidal category and $\D$ is a cocomplete monoidal category in which $(-) \otimes (-)$ preserves colimits in each variable (such as a cocomplete monoidal closed category), \emph{Day convolution} is the tensor product of a monoidal structure on $\Cat(\C,\D)$ defined by \[(F \otimes G)(c) = \int^{x, y \in \C} \C(x \otimes_\C y, c) \cdot (F(x) \otimes_\D G(y)).\]
This can be understood as defining $F \otimes G = \Lan{{\otimes}_\C} ({\otimes_\D} \circ (F \times G))$.
Moreover, the unit is given by $\Lan{I_\C}(I_\D)$ (where $I_\C$ and $I_\D$ are viewed as constant functors from $\1$).

We proved in \cref{thmtotalccliff} that $E\:\CInvMon\to \SLat$ is equivalent over $\SLat$ to the monoidal Grothendieck construction of the weakly lax monoidal pseudofunctor
\begin{fun}
\CInvMon_{(-)}^* & \: & \SLat & \too & \Cat \\[1ex]
&& L & \mto & \CInvMon_L\simeq \Abf(\Fib[d]_{L})\simeq \HomC{\Cat}{L\op}{\Ab}
\end{fun}
In this section, since the commutative inverse monoids are viewed additively, it will be convenient to work with \emph{join}-semilattices and thus take the reverse order on the Clifford monoids. With this convention we consider functors from $L$ to $\Ab$, not from $L\op$ to $\Ab$.

We now show that the monoidal structure on the fibres of $\CInvMon_{(-)}^*$ over idempotent semirings induced by \prox\ref{propinducepsmon} (see also \consx\ref{consinducemonstr}) is none other than the monoidal structure given by Day convolution.

\begin{theorem}\label{thm:DayConvolutionFibre}
    Consider the pseudofunctor $\widehat{\CInvMon_{(-)}^*}\:\Monf(\SLat)\to \MonCat$ induced by $\CInvMon_{(-)}^*$. For every $(Q,m,e)\in \Monf(\SLat)$, the monoidal structure on 
    $$\widehat{\CInvMon_{(-)}^*}(Q,m,e)=\HomC{\Cat}{Q}{\Ab}$$
    coincides with that of Day convolution.
\end{theorem}
\begin{proof}
    By \consx\ref{consinducemonstr}, the induced monoidal structure on $\HomC{\Cat}{Q}{\Ab}$ has tensor product given by
    $$\HomC{\Cat}{Q}{\Ab}\x \HomC{\Cat}{Q}{\Ab}\aar{\phi_{Q,Q}} \HomC{\Cat}{(Q\otimes Q)}{\Ab}\aar{\Lan{m}}\HomC{\Cat}{Q}{\Ab}$$
    and unit given by
    $$\1\aarr{\iota}\HomC{\Cat}{I}{\Ab}\aar{\Lan{e}}\HomC{\Cat}{Q}{\Ab}$$
    where $\phi$ and $\iota$ are the pseudonatural transformation given by the fact that $\CInvMon_{(-)}^*$ is a weakly lax monoidal pseudofunctor.
    
    As noted in \cref{consweaklylaxmonpseudofunstruct}, these pseudonatural transformations are given by the tensor product on $\Ab$ with $\phi_{Q,Q}(F,G) = \Lan{t} (\otimes_\Ab \circ (F \times G))$ where $t\colon Q \times Q \to Q \otimes Q$ sends $(a,b)$ to $a \otimes b$. Above we view $m$ as a map from $Q \otimes Q$ to $Q$, but the multiplication of $Q$ can also be considered to be $mt\colon Q \times Q \to Q$. Then the tensor product defined above is given by $F \otimes G = \Lan{m}(\Lan{t} (\otimes_\Ab \circ (F \times G))) \cong \Lan{mt}(\otimes_\Ab \circ (F \times G))$. This is precisely the definition of the Day convolution on $\Cat(Q, \Ab)$.
    The case of the unit is similar.
\end{proof}
\begin{remark}
    More generally, a similar result should be true if $\SLat$ is replaced by the 2-category of small categories and $\Ab$ is replaced with any cocomplete monoidal closed category, but making complete sense of this would require a higher-dimensional version of the monoidal Grothendieck construction.
\end{remark}

We will use the following folklore result about monoids with respect to Day convolution.
\begin{proposition}\label{prop:lax_monoidal_monoid_wrt_Day}
    Let $\C$ be a small monoidal category and let $\D$ be a cocomplete monoidal closed category.
    Then monoids with respect to the monoidal structure on the functor category $\Cat(\C,\D)$ given by Day convolution are precisely the lax monoidal functors from $\C$ to $\D$.
\end{proposition}
\begin{proof}
    Consider an object $F \in \Cat(\C,\D)$. A monoid structure on $F$ consists of a unit $\iota\: \Lan{I_\C}(I_\D) \to F$ and a multiplication $\mu\: \Lan{{\otimes}_\C} ({\otimes_\D} \circ (F \times F)) \to F$. Using that $\Lan{{\otimes}_\C}$ is left adjoint to composition with ${\otimes}_\C$, we see that these correspond to a morphism $\iota^\sharp\: I_\D \to F(I_\C)$ and a natural transformation $\mu^\sharp\: {\otimes_\D} \circ (F \times F) \to F \circ {\otimes}_\C$ or $\mu^\sharp_{C,C'}\colon F(C) \otimes F(C') \to F(C \otimes C')$, respectively. These will be the unit and multiplication maps for the lax monoidal structure on $F$.
    Finally, the unit and associativity laws for the monoid correspond to the coherence laws for the resulting lax monoidal functor.
\end{proof}

We now obtain our structure theorem for inverse semirings.
\begin{theorem}
    The category of inverse semirings is equivalent over the category of idempotent semirings to the Grothendieck construction of the pseudofunctor
    $$\Monf(\SLat)\aar{\widehat{\CInvMon_{(-)}^*}} \MonCat\aar{\Monf(-)}\Cat$$
    which takes monoids in the fibres $\HomC{\Cat}{Q}{\Ab}$ with respect to the Day convolution.

    In particular, inverse semirings form a (Street) opfibration over idempotent semirings.
\end{theorem}
\begin{proof}
    We apply the general \thex\ref{theormoninmonoidalgroth} to the weakly lax monoidal pseudofunctor $\CInvMon_{(-)}^*\:\SLat\to \Cat$, whose monoidal Grothendieck construction gives a monoidal opfibration equivalent to $E\:\CInvMon\to \SLat$.
    We obtain that
    $$\Monf(\CInvMon)\simeq \Monf(\Grothdiag{\CInvMon_{(-)}^*})\simeq \Grothdiag{(\Monf(-)\c \widehat{\CInvMon_{(-)}^*})}$$
    over $\Monf(\SLat)$.
    But $\Monf(\CInvMon)$ is just the category of inverse semirings and $\Monf(\SLat)$ is the category of idempotent semirings.
\end{proof}

Finally, combining this with \cref{prop:lax_monoidal_monoid_wrt_Day} we obtain the following explicit correspondence.
\begin{theorem}\label{thm:inverse_semiring_structure_theorem}
    Inverse semirings precisely correspond to idempotent semirings $Q$ equipped with a lax monoidal functor $Q\to \Ab$.
    
    Inverse semiring homomorphisms correspond to pairs $(f,\alpha)$ where $f$ is a semiring homomorphism between the respective idempotent semirings $Q$ and $Q'$ and $\alpha$ is a natural transformation living in $\HomC{\Cat}{Q}{\Ab}$.

    Given such a functor, the associated inverse semiring consists of elements $(e,g)$ with $g \in F(e)$. The addition is given by $(e_1,g_1) + (e_2,g_2) = (e_1 \vee e_2, F(e_1 \le e_1 \vee e_2)(g_1)+F(e_2 \le e_1 \vee e_2)(g_2))$ with additive inverse $(0,0)$.
    
    Let $\mu_{X,Y}\colon F(X) \otimes F(Y) \to F(X \otimes Y)$ be the associated `multiplication' structure 2-cell of $F$. Then the multiplication is given by $(e_1,g_1)(e_2,g_2) = (e_1e_2,\mu(g_1 \otimes g_2))$ and multiplicative identity given by $(1,\iota(1))$, where in the first component $1$ is the multiplicative unit of $Q$ and in the second, it is the well-known integer.
\end{theorem}

\begin{remark}
 Here we have come full circle: considering only the multiplicative structure, this is now again an example of the `total face' of the monoidal Grothendieck construction.
 (On the additive part of this structure, this is simply the correspondence for commutative Clifford monoids and is hence obtained from the `fibrewise face' of the monoidal Grothendieck construction.)
 Thus, in order to move from the lax monoidal functor $F$ to its associated inverse semiring, one must combine the two faces of the monoidal Grothendieck construction.

This correspondence can likely be generalised to an as yet undefined notion of \emph{Janusian} Grothendieck construction. In particular, one could consider a lax monoidal functor $F\colon \C \to \SymMonCat$ where $\C$ is a rig category whose additive structure is cocartesian. Then applying this Grothendieck construction would yield an opfibration whose domain is a rig category.
\end{remark}

\begin{remark}
 Construction 3.5 of \cite{faul2026} arises as the special case of \cref{thm:inverse_semiring_structure_theorem} where $Q$ is the distributive lattice $\mathbbm{2}$.
\end{remark}

\bibliographystyle{abbrv}
\bibliography{bibliography}

\end{document}